\documentclass{amsart}

\usepackage{arydshln}
\usepackage{graphicx}
\usepackage[latin1]{inputenc}
\usepackage{tikz-cd}

\newtheorem{theorem}{Theorem}[section]
\newtheorem{thm}{Theorem}[section]

\newtheorem{lemma}[thm]{Lemma}
\newtheorem{proposition}[thm]{Proposition}
\theoremstyle{example}
\newtheorem{example}[thm]{Example}
\theoremstyle{definition}
\newtheorem{definition}[thm]{Definition}
\theoremstyle{remark}
\newtheorem{remark}[thm]{Remark}
\numberwithin{equation}{section}
\newcommand{\Real}{\mathbb R}

\newcommand{\Neal}{\mathbb N}

\begin{document}

\title[Invariants for surface-links in entropic magmas]
{On invariants for surface-links in entropic magmas via marked graph diagrams}

\author[S. Choi and S. Kim]{Seonmi Choi and Seongjeong Kim}

\address{Department of Mathematics, Kyungpook National University, Daegu, Korea\\
csm123c@gmail.com\\
Department of Mathematics, Jilin University, Changchun, China \\
kimseongjeong@jlu.edu.cn}


\begin{abstract}
M. Niebrzydowski and J. H. Przytycki defined a Kauffman bracket magma and constructed the invariant $P$ of framed links in $3$-space. The invariant is closely related to the Kauffman bracket polynomial. 
The normalized bracket polynomial is obtained from the Kauffman bracket polynomial by the multiplication of indeterminate 
and it is an ambient isotopy invariant for links. 

In this paper, 
we reformulate the multiplication by using a map from the set of framed links to a Kauffman bracket magma 
in order that $P$ is invariant for links in $3$-space.
We define a generalization of a Kauffman bracket magma, which is called a {\it marked Kauffman bracket magma}.
We find the conditions to be invariant under Yoshikawa moves except the first one 
and  
use a map from the set of admissible marked graph diagrams to a marked Kauffman bracket magma
to obtain the invariant for surface-links in $4$-space.
\end{abstract}

\maketitle

\section{Introduction}

The Jones polynomial introduced by V. F. R. Jones in 1984 \cite{Jones} is considered as one of revolutions in Knot theory. 
At first it was constructed by the Artin braid groups and their Markov equivalence by means of a Markov trace on the Temperley-Lieb algebras. 
Simultaneously, it was shown that the Jones polynomial can be calculated by using a skein relation.

In 1985, 
the 2-variable polynomial invariant, called the {\it HOMFLY polynomial}, 
was discovered by J. Hoste, A. Ocneanu, K. Millett, P. J. Freyd, W. B. R. Lickorish and D. N. Yetter 
and this invariant generalizes both the Alexander-Conway polynomial and the Jones polynomial.
They suggested the several approaches to construct the polynomial 
by using the Ocneanu trace defined on the Iwahori-Hecke algebra of type A and a skein relation in \cite{HOMFLY}.
At the same time, J. H. Przytycki and P. Traczyk defined new algebraic structure, called a {\it Conway algebra} in \cite{PrzytyckiTraczyk} and constructed invariants of links valued in a Conway algebra.
It is noticed that the HOMFLY polynomial can be obtained from this invariant 
and it is called the {\it HOMFLYPT polynomial} with the recognition of the work of J. H. Przytycki and P. Traczyk in \cite{PrzytyckiTraczyk}.

On the other hand, the state-sum model is proposed by L. H. Kauffman \cite{Kauffman}, which is known as the {\it Kauffman bracket} in 1987.
The Kauffman bracket polynomial gives a regular isotopy invariant for links. 
The Jones polynomial can be obtained from the Kauffman bracket polynomial by multiplying indeterminate with respect to the writhe of a given link diagram. 
Notice that the multiplication of indeterminate cancel the differences coming from Reidemeister move I.

In 2013, M. Niebrzydowski and J. H. Przytycki \cite{NiebrzydowskiPrzytycki} suggested an invariant constructed via an entropic magma with a sequence satisfying some properties, which is called a {\it Kauffman bracket magma}, and the Kauffman bracket polynomial is one of invariants valued in the Kauffman bracket magma.


A {\it surface-link} is a closed surface embedded in $4$-space and 
it is not easy to imagine surface-links directly.
The {\it marked graph diagram} was introduced by 
S. J. Lomonaco, Jr. \cite{Lomonaco} and K. Yoshikawa \cite{Yoshikawa} to describe surface-links in $4$-space.
It is known that for a given admissible marked graph diagram $D$, one can construct a surface-link $F(D)$
and conversely, every surface-link $F$ can be represented by an admissible marked graph diagram $D$
such that the surface-link $F(D)$ is equivalent to $F$.
K. Yoshikawa introduced local moves 
$\Gamma_{1}, \Gamma_{2}, \Gamma_{3}, \Gamma_{4}, \Gamma_{4}', \Gamma_{5}, \Gamma_{6}, \Gamma_{6}', \Gamma_{7}$ and $\Gamma_{8}$
on marked graph diagrams in \cite{Yoshikawa}, which are called {\it Yoshikawa moves}.
It is known that two marked graph diagrams present the same surface-link if and only if they are related by a finite sequence of Yoshikawa moves.
Then one can construct some invariants for surface-links by applying the construction of invariants for links via marked graph diagrams.


In 2008, S. Y. Lee \cite{SYLee} introduced a method of constructing ambient isotopy invariants for surface-links in $4$-space by using the description of marked graph diagrams. The ambient isotopy invariants for surface-links were constructed from ambient or regular isotopy invariants for links in $3$-space. 


In this paper, 
we deal with skein relations to construct invariant for links and surface-links valued in an {\it entropic} magma. 
The entropic condition is important in the theory of polynomial invariants.

This paper consists of two parts : 
the first part focus on invariants for links in 3-space valued in a Kauffman bracket magma.
More precisely, the Kauffman bracket magma in \cite{NiebrzydowskiPrzytycki} provides the regular isotopy invariant $P$ 
and one of its examples is the Kauffman bracket polynomial. 
It is known that the Kauffman bracket polynomial can be ambient isotopy invariant 
by multiplying additional indeterminate with respect to a given diagram. 
In this part, we reformulate the multiplication of indeterminate 
by using a map $\rho$ from the set of framed links to the Kauffman bracket magma 
and its ``action'' on the invariant $P$ (Theorem \ref{Thm0}).

The second part deals with invariants for surface-links in 4-space values in a generalization of a Kauffman bracket magma, which is called a {\it marked Kauffman bracket magma}, by using marked graph diagrams.
First of all, we construct a function $P_{M}$ from the set $\mathcal{F}$ of admissible marked graph diagrams to a marked Kauffman bracket magma and then determine the conditions for $P_{M}$ to be invariant under Yoshikawa moves except $\Gamma_{1}$.
Similar to Theorem \ref{Thm0}, 
by using a map $\rho$ from $\mathcal{F}$ to the marked Kauffman bracket magma in Theorem \ref{Thm3},
we can obtain the ambient isotopy invariant for surface-links.
In addition, we apply the same method to construct the invariant under $\Gamma_{6}$ and $\Gamma_{6}'$, 
instead of the additional conditions determined by them.

In Section 2, we mention some polynomial invariants for links and 
review basic definitions and results related to surface-links and marked graph diagrams with Yoshikawa moves.
In the end of the section, we introduce a polynomial invariant via marked graph diagrams.
In Section 3, we introduce an invariant for framed links valued in a Kauffman bracket magma 
and modify the invariant to be an ambient isotopy invariant by using a map $\rho$.
In Section 4, we discuss on invariants for surface-links valued in a marked Kauffman bracket magma via marked graph diagrams.
In Section 4.1, we define a marked Kauffman bracket magma and construct an invariant $P_{M}$ under Yoshikawa moves except $\Gamma_{1}, \Gamma_{6}$ and $\Gamma_{6}'$. 
In Section 4.2, we find the conditions in which $P_{M}$ is invariant under $\Gamma_{6}$ and $\Gamma_{6}'$ 
and modify the invariant to be an ambient isotopy invariant by using a map $\rho$.
There are some examples corresponding to theorems in Section 4.3.
We discuss on another way to obtain invariants under $\Gamma_{6}$ and $\Gamma_{6}'$ 
by using a map $\phi$ in Theorem \ref{Thm4} analogously to the case of $\Gamma_{1}$ in Section 4.4.


\section{Preliminaries}

\subsection{Polynomial invariants for links}

There are many polynomial invariants for knots and links in $\mathbb{R}^{3}$.

\begin{definition}[\cite{Jones}]
The {\it Jones polynomial $V_{D}(t)$} of a diagram $D$ of an oriented link $L$ is a polynomial in $\mathbb{Z}[A^{\pm 1}]$ satisfying
\begin{enumerate}
  \item $V(O) =1$ for the standard diagram $O$ of the unknot,
  \item $ t^{-1}V_{D_{+}}(t)-tV_{D_{-}}(t)=(t^{\frac{1}{2}}-t^{- \frac{1}{2}})V_{D_{0}}(t)$,
\end{enumerate}
where ($D_{+},D_{-},D_{0}$) is the Conway skein triple as described in Fig. \ref{SkeinConway(ori)}.
\begin{figure}[h!]
 \centering
 \includegraphics[width = 12cm]{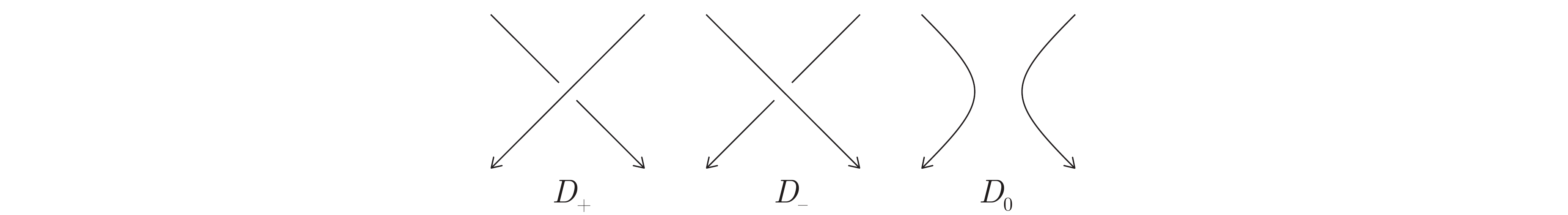}
 \caption{The Conway skein triple ($D_{+},D_{-}$, $D_{0}$)}\label{SkeinConway(ori)}
\end{figure}
\end{definition}

The polynomial $V_{D}(t)$ is an ambient isotopy invariant and denoted by $V_{L}(t)$.
In \cite{Kauffman}, the Jones polynomial was defined by using the Kauffman bracket polynomial.

\begin{definition}[\cite{Kauffman}]
The {\it Kauffman bracket polynomial $\langle D \rangle$} of a diagram $D$ of any unoriented link $L$ 
is a polynomial in $\mathbb{Z}[A^{\pm 1}]$ such that
\begin{enumerate}
\item $\langle O \rangle =1$ for the standard diagram $O$ of the unknot,
\item $\langle O \sqcup D \rangle = (-A^{2} - A^{-2})\langle D \rangle$,
\item $\langle D \rangle = A\langle D_{0} \rangle+A^{-1}\langle D_{\infty} \rangle$,
\end{enumerate}
where ($D$, $D_{0}$, $D_{\infty}$) is the Kauffman skein triple as described in Fig. \ref{SkeinKauffman}.
\begin{figure}[h!]
 \centering
 \includegraphics[width = 12cm]{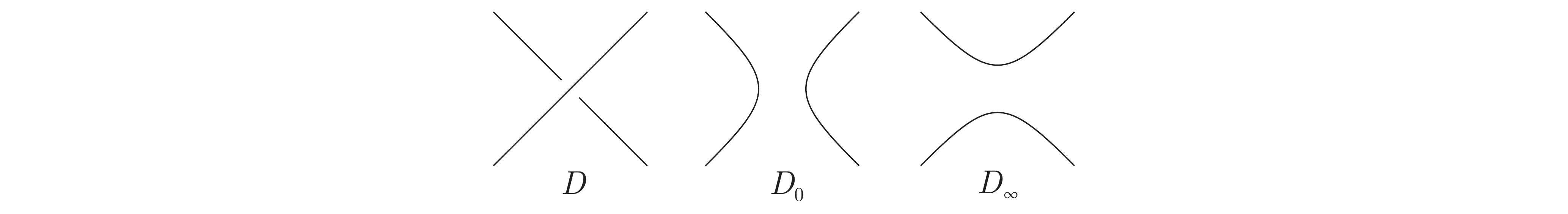}
 \caption{The Kauffman skein triple ($D$, $D_{0}$, $D_{\infty}$)}\label{SkeinKauffman}
\end{figure}
\end{definition}

The bracket polynomial $\langle D \rangle$ is a regular isotopy invariant (i.e. invariant under Reidemeister moves II and III).
That is, it is an invariant for framed links. 
Define a polynomial $X(D)$ by $X(D) = (-A^{3})^{-w(D)}\langle |D|\rangle$, where $w(D)$ is the writhe of $D$ and $|D|$ is the diagram obtained from $D$ forgetting a given orientation.
It is invariant under all Reidemeister moves and denoted by $X(L)$, which is called the {\it normalized bracket polynomial} of $L$.

\begin{remark}
It is well-known that $X(D)|_{A=t^{- \frac{1}{4}}}= V_{D}(t)$. 
Let us notice that the Kauffman bracket polynomial is defined by ``splicings'' of crossings and is a regular isotopy invariant. 
By multiplying $(-A^{3})^{-w(D)}$, one can obtain an ambient isotopy invariant $X(D)$ for links
from the regular isotopy invariant $\langle ~ \rangle$.
In fact the multiplication of $(-A^{3})^{-w(D)}$ is a homomorphism of the Laurent polynomial ring $\mathbb{Z}[A^{\pm 1}]$.
In addition, by changing a variable, we can obtain the Jones polynomial, which is constructed by applying the Conway skein triple.
\end{remark}

\subsection{Surface-links and marked graph diagrams}

Now we review marked graph diagrams as a tool for studying surface-links. 
A {\it surface-link} is a smoothly embedded closed surface  in $\Real^{4}$.
If it is connected, then it is called a {\it surface-knot}.
Two surface-links are said to be {\it equivalent} if they are ambient isotopic in $\Real^{4}$.
A surface-link is called {\it orientable} if the underlying surface is orientable.
Otherwise, we call it {\it non-orientable}.
To deal with surface-links, there are many descriptions for surface-links and see the details in \cite{CarterKamadaSaito,Kamadabook,Kamadabook2,Yoshikawa}. 

In this section, we introduce one of these descriptions for surface-links named by a {\it marked graph diagram} with {\it Yoshikawa moves}.
A {\it marked graph} is a $4$-valent graph embedded in $\Real^3$ such that each vertex is decorated by a line segment, called a {\it marker} as shown in Fig. \ref{MarkedVertex}. (A). 
A vertex with a marker is called a {\it marked vertex}.

\begin{figure}[h!]
 \centering
 \includegraphics[width = 12cm]{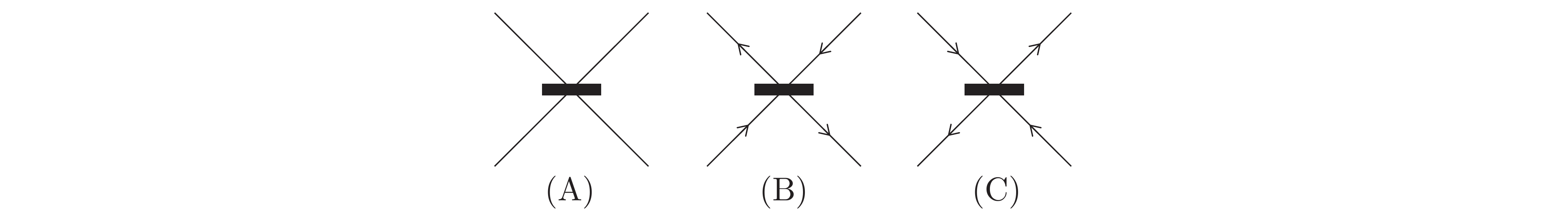}
 \caption{A marked vertex and its orientation}\label{MarkedVertex}
\end{figure}
An orientation of edges incident with a marked vertex is a choice of an orientation as described in Fig. \ref{MarkedVertex} (B) and (C).
A marked graph is {\it orientable} if it admits an orientation.
Otherwise, it is {\it non-orientable}. 
Two (oriented) marked graphs are said to be {\it equivalent} if they are ambient isotopic in $\Real^3$ with keeping the rectangular neighborhoods and markers (with orientation).
A {\it marked graph diagram} can be described by a diagram in $\Real^2$, which is a link diagram with marked vertices.

For a marked graph diagram $D$, 
let $L_{-}(D)$ and $L_{+}(D)$ denote classical link diagrams obtained from $D$ by changing every marked vertex to local diagrams illustrated in Fig. \ref{MarkedVertexSplicing} (A) and (C), respectively.
\begin{figure}[h!]
 \centering
 \includegraphics[width = 12cm]{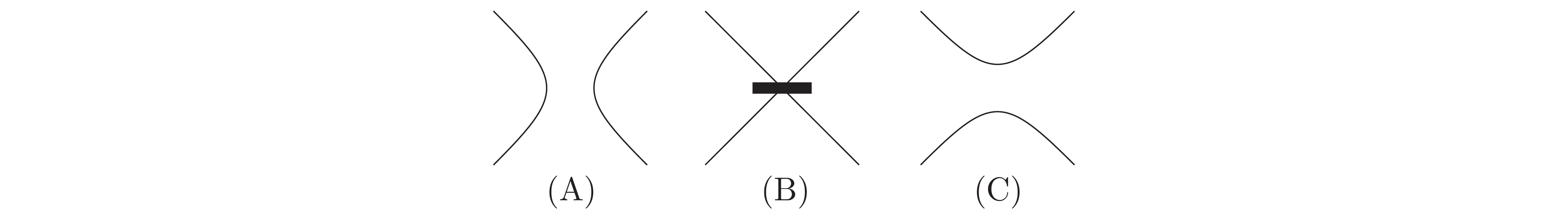}
 \caption{Splicings at a marked vertex}\label{MarkedVertexSplicing}
\end{figure}
Two link diagrams $L_{-}(D)$ and $L_{+}(D)$ are called the {\it negative} and {\it positive resolution} of $D$, respectively. 
A marked graph diagram $D$ is said to be {\it admissible} if both resolutions $L_{-}(D)$ and $L_{+}(D)$ are trivial link diagrams. A marked graph is said to be {\it admissible} if it has an admissible marked graph diagram.
For example, Fig. \ref{MarkedGraph} describes a marked graph diagram of the spun trefoil, its positive and negative resolutions. One can easily check that $L_{+}(D)$ and $L_{-}(D)$ described in Fig. \ref{MarkedGraph} are trivial link diagrams and hence the marked graph diagram of the spun trefoil is admissible.
\begin{figure}[h!]
 \centering
 \includegraphics[width = 12cm]{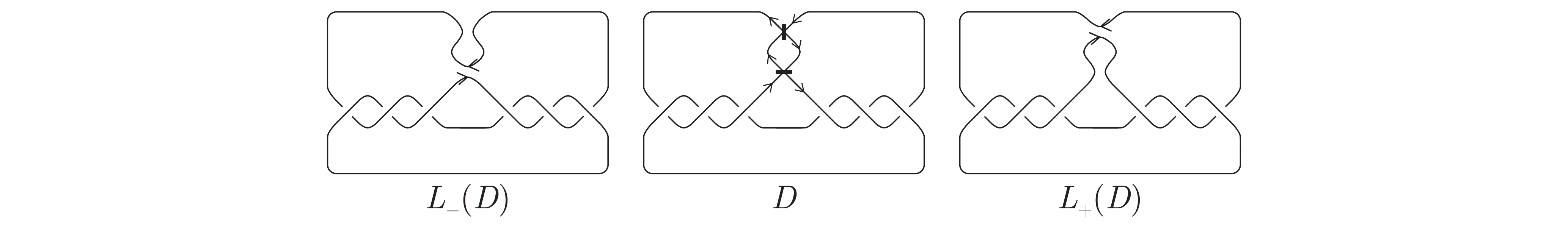}
 \caption{A marked graph diagram of the spun trefoil and its resolutions}\label{MarkedGraph}
\end{figure}

From a given admissible marked graph diagram $D$, 
one can construct a surface-link $F(D)$ and it is uniquely determined from $D$ up to equivalence. 
Conversely, every surface-link $F$ can be represented by an admissible marked graph diagram $D$, that is, $F(D)$ is equivalent to $F$.
See the details in \cite{KawauchiShibuyaSuzuki,Lomonaco,Yoshikawa}. 
For example, the correspondence between the marked graph diagram and the standard projective plane are illustrated in Fig. \ref{RP2}.
\begin{figure}[h!]
 \centering
 \includegraphics[width = 12cm]{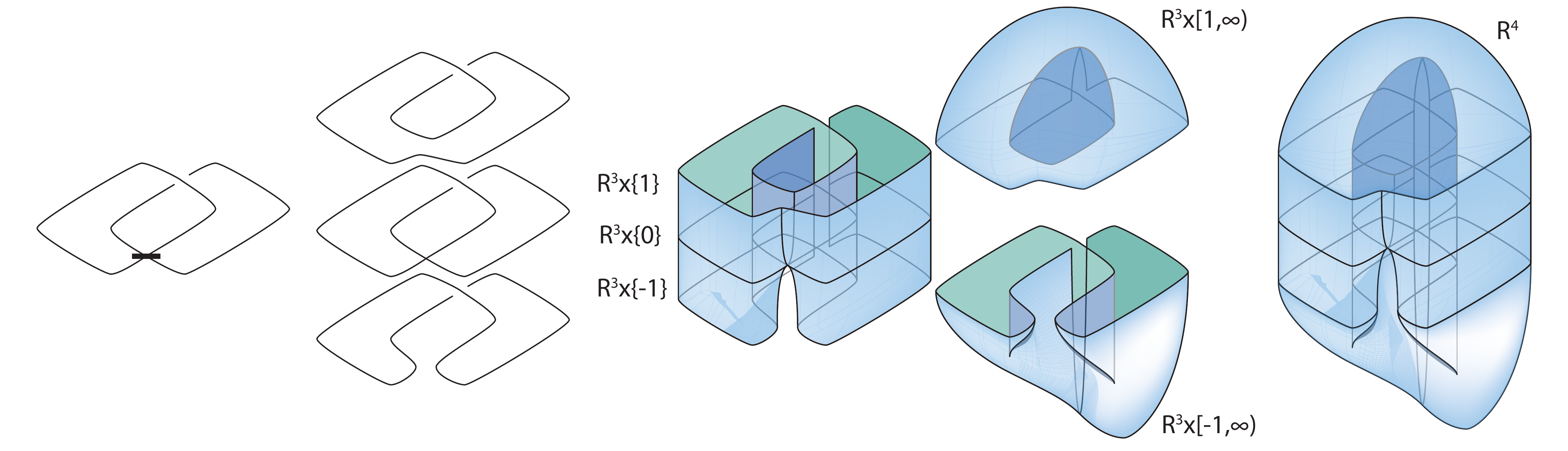}
 \caption{The standard projective plane}\label{RP2}
\end{figure}
There are local moves on marked graph diagrams introduced by K. Yoshikawa \cite{Yoshikawa}, which are called {\it Yoshikawa moves} as depicted in Fig. \ref{UnoriYoshikawaMoves}.
\begin{figure}[h!]
 \centering
 \includegraphics[width = 12cm]{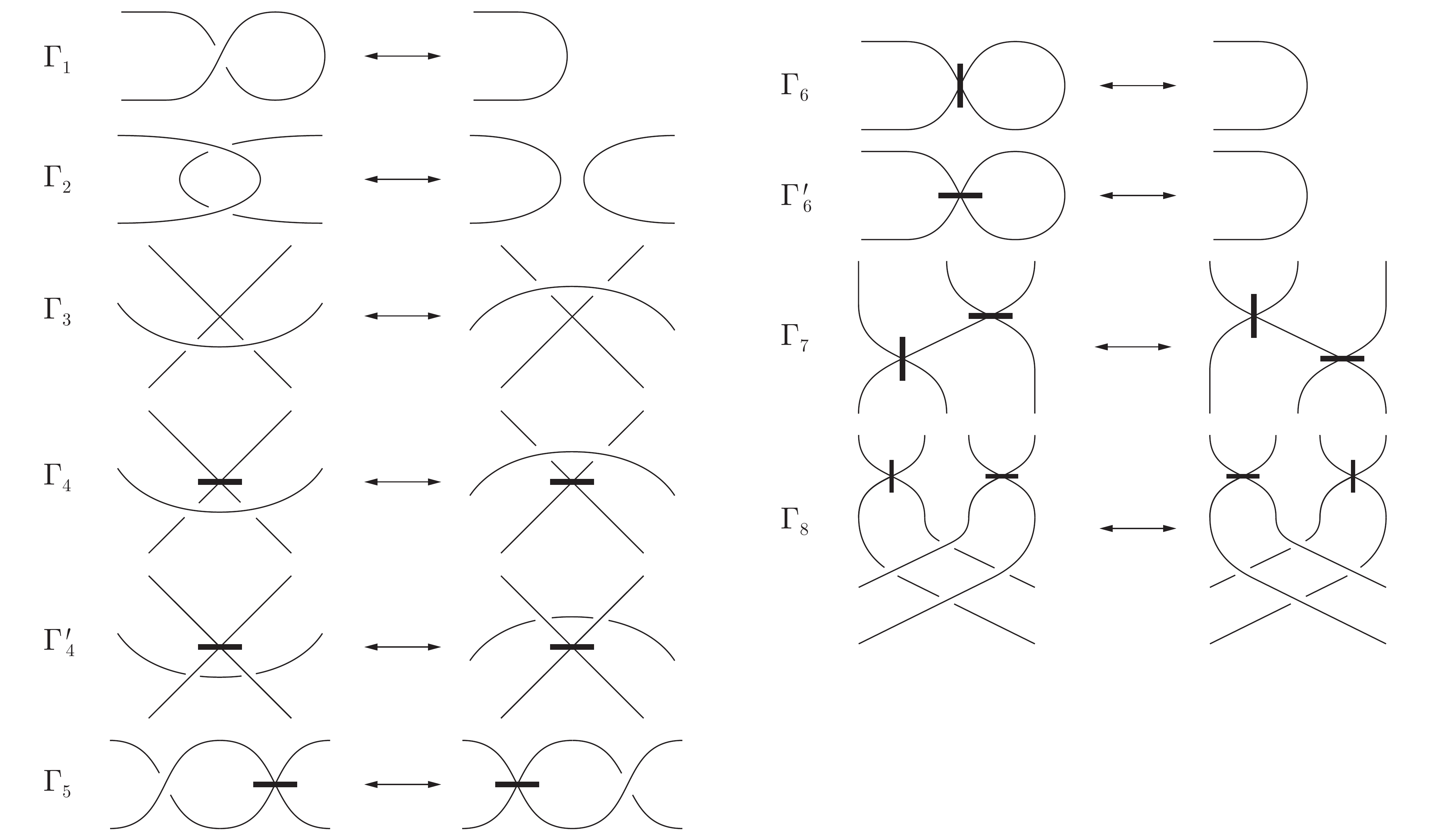}
 \caption{Yoshikawa moves}\label{UnoriYoshikawaMoves}
\end{figure}
Yoshikawa moves of {\it type $1$} (resp.  {\it type $2$}) consists of $\Gamma_{1}, \Gamma_{2}, \cdots, \Gamma_{5}$. 
(resp. $\Gamma_{6}, \Gamma_{6}', \Gamma_{7}, \Gamma_{8}$.)
Two marked graph diagrams present equivalent oriented marked graphs if and only if they are related by a finite sequence of Yoshikawa moves of type $1$.
\begin{proposition}[\cite{KeartonKurlin,Swenton,Yoshikawa}]
Two marked graph diagrams $D$ and $D'$ present equivalent oriented surface-links if and only if $D$ can be obtained from $D'$ by a finite sequence of ambient isotopies in $\mathbb{R}^{2}$ and Yoshikawa moves.
\end{proposition}

\subsection{Polynomial invariants for surface-links}

In 2008, S.Y. Lee \cite{SYLee} introduced a construction of ambient isotopy invariants for surface-links in $4$-space 
by using marked graph diagrams and an arbitrary given isotopy or regular isotopy invariant of classical links in $3$-space. 

For a commutative ring $R$ with the additive identity $0$ and the multiplicative identity $1$,
let $\hat{R}$ denote the polynomial ring $R[A_{1}, \cdots, A_{m}]$ ($m \geq 0$).
If $m=0$, then $\hat{R}=R$.
Let $[~ ]$ be a regular or an ambient isotopy invariant of classical links in $3$-space with the values in $\hat{R}$
satisfying the following conditions :
for a classical link diagram $D$ 
and
the standard diagram $O$ of the unknot,
\begin{enumerate}
  \item $[D_{+}]=\alpha [D]$ and $[D_{-}]=\alpha [D]$, 
  \item $[D\sqcup O]=\delta[D]$, 
\end{enumerate}
where 
$D_{+}$ and $D_{-}$ are diagrams obtained from $D$ by adding a positive and negative kink respectively, depicted in Fig. \ref{RM1}, 
$\delta$ is an element $\delta (A_{1}, \cdots, A_{m})\in \hat{R}$ 
and 
$\alpha$ is an invertible element $\alpha(A_{1}, \cdots, A_{m})\in \hat{R}$.
\begin{figure}[h!]
 \centering
 \includegraphics[width = 12cm]{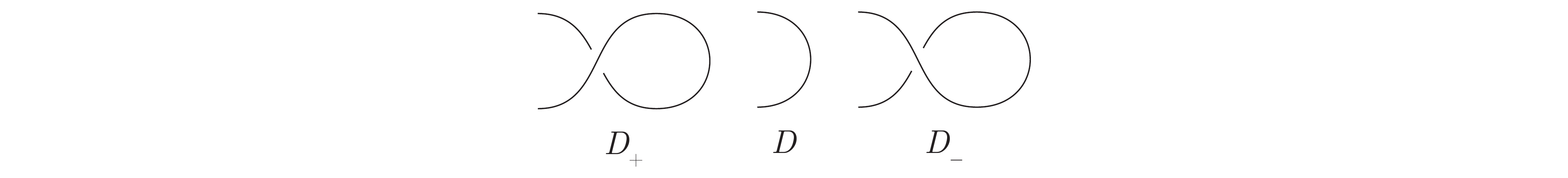}
 \caption{The diagrams of $D_{+}$, $D$ and $D_{-}$}\label{RM1}
\end{figure}

Notice that $[~]$ is an ambient isotopy invariant of classical knots and links if and only if $\alpha=1$.

\begin{definition}[\cite{SYLee}]
Let $D$ be a marked graph diagram. 
Define a polynomial $[[D]]=[[D]](A_{1},\cdots,A_{m}, x, y, z, w)$ in $\hat{R}[x, y, z, w]$ by means of the two rules :
\begin{enumerate}
  \item $[[D]]=[D]$ if $D$ is a classical knot or link diagram.
  \item $[[D^{m}]]=x[[D_{\infty}^{m}]]+y[[D_{+}^{m}]]+z[[D_{-}^{m}]]+w[[D_{0}^{m}]]$ 
\end{enumerate}
where $D^{m}, D_{\infty}^{m}, D_{+}^{m}, D_{-}^{m}$ and $D_{0}^{m}$ are the local diagrams that are identical except the small parts, as depicted in Fig. \ref{SkeinMarkedGen}.

\begin{figure}[h!]
 \centering
 \includegraphics[width = 12cm]{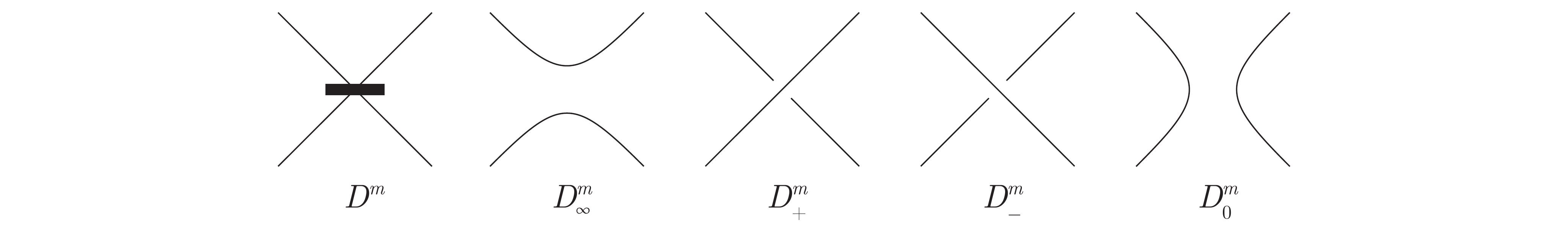}
 \caption{The local diagrams of $D^{m}, D_{\infty}^{m}, D_{+}^{m}, D_{-}^{m}$ and $D_{0}^{m}$}\label{SkeinMarkedGen}
\end{figure}

\end{definition}

\begin{proposition}[\cite{SYLee}]
For any regular (resp. ambient ) isotopy invariant of classical links, 
$[[D]]$ is invariant under Yoshikawa moves $\Gamma_{2}, \cdots, \Gamma_{5}$ (resp. $\Gamma_{1}, \Gamma_{2}, \cdots, \Gamma_{5}$).
\end{proposition}

By using this construction, one can obtain the polynomial $[[ ~ ]]$ from the Kauffman bracket polynomial $\langle ~ \rangle$, 
as an example.

\begin{example}[\cite{SYLee}]\label{ExaSYLee}
Let $R=\mathbb{Z}[A^{\pm 1}, x, y, z, w]$ denote the polynomial ring in the commutative variables $A, x, y, z, w$ with coefficient in $\mathbb{Z}$. 
For a marked graph diagram $D$,
construct the polynomial $[[D]]$ such that
\begin{enumerate}
  \item $[[D]]=\langle D\rangle (A)$ if $D$ is a classical knot or link diagram,
  \item $[[D_{m}]]=(x+Ay+A^{-1}z)[[D_{\infty}^{m}]]+(A^{-1}y+Az+w)[[D_{0}^{m}]]$.
\end{enumerate}
Then the polynomial $[[D]]$ is invariant under $\Gamma_{2}, \cdots, \Gamma_{5}$.
\end{example}

In order to be invariant under all Yoshikawa moves, 
the polynomials were defined by using $[[D]]$, some new variables and functions.
In particular, there is the calculation for a polynomial corresponding to Example \ref{ExaSYLee}.
For the details, see Theorem 4.3 and Theorem 4.4 in \cite{SYLee}.


\section{Kauffman bracket magmas and their invariants}

\subsection{Kauffman bracket magma and its invariant}

In 2012, M. Niebrzydowski and J. H. Przytycki introduced an algebraic structure, which is called a {\it Kauffman bracket magma} and constructed unoriented framed link invariants. 

\begin{definition}[\cite{Sushkevich}]\label{def-Entropic}
Let $A$ be a set equipped with a binary operation $*$. 
If for all $a, b, c, d \in A$, $(a*b)*(c*d)=(a*c)*(b*d)$, then the pair $(A;*)$ is said to be {\it entropic}.
\end{definition}

The entropic property was used in Knot theory in order to compute some invariants; for examples, Jones and HOMFLY-PT polynomial invariants. 

\begin{definition}[\cite{NiebrzydowskiPrzytycki}]\label{def-KBA}
Let $(A;*)$ be an entropic magma. 
Take a sequence $\{a_{n}\}_{n=1}^{\infty}$ of not necessarily different elements in $A$ satisfying the following conditions : for all $n \in \mathbb{N}$,
\begin{enumerate}
\item $(a_{n+1}*a_{n+2})*(a_{n}*a_{n+1})=a_{n}$,
\item $(a_{n}*a_{n+1})*(a_{n+1}*a_{n})=a_{n+1}$.
\end{enumerate}
Then $(A;*,\{a_{n}\})$ is called a {\it Kauffman bracket magma}.
\end{definition}

The two conditions (1) and (2) of the definition of a Kauffman bracket magma are derived from the Reidemeister move II by the  numerator and denominator closures, respectively.

\begin{figure}[h!]
 \centering
 \includegraphics[width = 12cm]{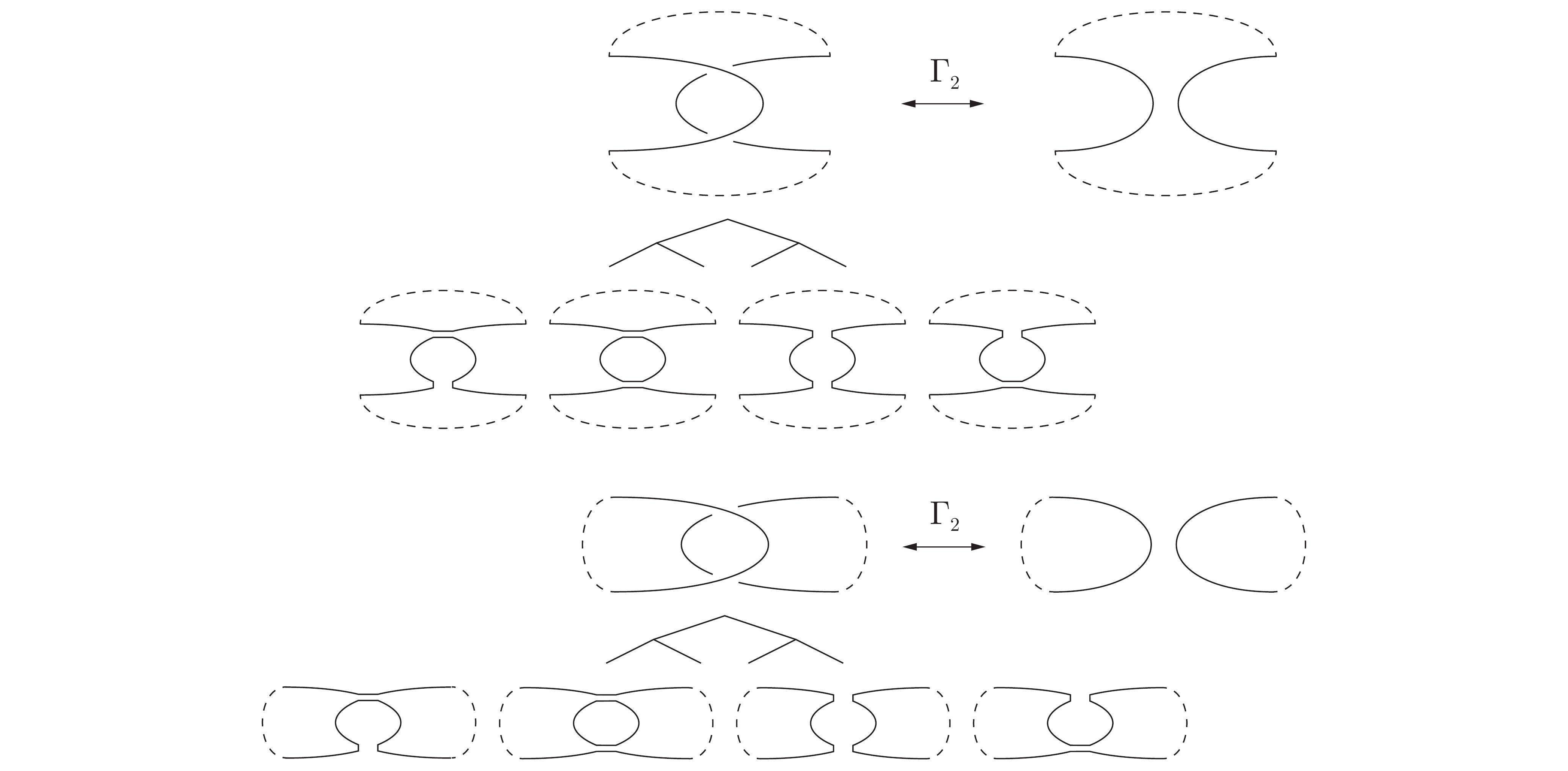}
 \caption{The numerator and denominator closures of Reidemeister move II}
\end{figure}

\begin{proposition}[\cite{NiebrzydowskiPrzytycki}]
Let ${\mathcal L^{fr}}$ be the set of equivalence classes of unoriented links under Reidemeister moves II and III.
For a given Kauffman bracket magma $(A;*,\{a_{n}\})$, there is a unique invariant $P : {\mathcal L^{fr}} \rightarrow A$ for unoriented framed links such that
  \begin{enumerate}
    \item $P(T_{n})=a_{n}$, 
    \item $P(D)=P(D_{0})*P(D_{\infty})$
  \end{enumerate}
where $T_{n}$ is the trivial link diagram with $n$ components 
and $(D, D_{0}, D_{\infty})$ is the Kauffman skein triple as depicted in Fig. \ref{SkeinKauffman}.
\end{proposition}

\begin{example}\label{ExaKB}
Let $\mathbb{Z}[p,r]$ be a polynomial ring equipped with a binary operation $*$ defined by 
$$a*b= pa+(-1-p)b+r.$$
For a fixed element $c \in \mathbb{Z}[p,r]$, take a sequence $a_{n} = c$ for each $n\in \mathbb{N}$.
Then the triple $(\mathbb{Z}[p,r]; *, \{a_{n}\} )$ is a Kauffman bracket magma.
\end{example}

The Kauffman bracket polynomial is obtained from one of examples of Kauffman bracket magmas as follows.

\begin{example}\label{ExaKauffmanBracketInv}
Let $\mathbb{Z}[A^{\pm 1}]$ be a polynomial ring equipped with a binary operation $*$ defined by 
$$a*b= Aa+A^{-1}b.$$
Consider the sequence $a_{n}=(-A^{2}-A^{-2})^{n-1}$ for each $n\in \mathbb{N}$.
Then the triple $(\mathbb{Z}[A^{\pm 1}]; *, \{a_{n}\} )$ is a Kauffman bracket magma.
Hence one can obtain the invariant $P$ of unoriented framed links and it is exactly equal to the Kauffman bracket polynomial. 
\end{example}

For instance, the calculation of the invariant $P$ of the trefoil $3_1$ is 
$P(3_{1})=[(P(K_{1})*P(K_{2}))*(P(K_{3})*P(K_{4}))]*[(P(K_{5})*P(K_{6}))*(P(K_{7})*P(K_{8}))]$
$=[(a_{2}*a_{1})*(a_{1}*a_{2})]*[(a_{1}*a_{2})*(a_{2}*a_{3})]$
and its resolving tree is depicted in Fig. \ref{pTrefoilKauffmanInv}.
  \begin{figure}[h!]
    \centering
    \includegraphics[width = 12cm]{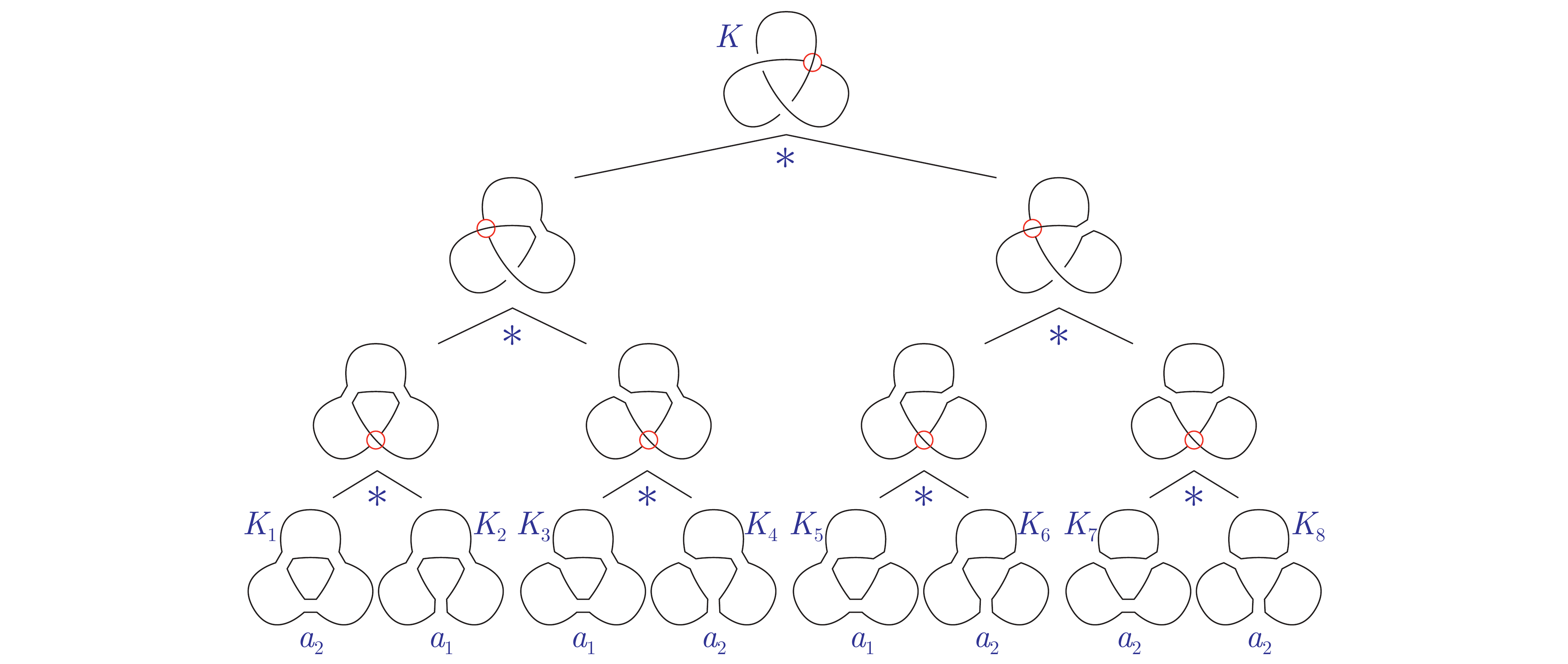}
    \caption{The resolving tree of $3_{1}$} \label{pTrefoilKauffmanInv}
  \end{figure}
In the case of the Kauffman bracket magma in Example \ref{ExaKB}, 
$P(3_{1})=[(c*c)*(c*c)]*[(c*c)*(c*c)]=-c+r$. 
By using the Kauffman bracket magma in Example \ref{ExaKauffmanBracketInv}, 
$P(3_{1})=[((-A^{2}-A^{-2})*1)*(1*(-A^{2}-A^{-2}))]*[(1*(-A^{2}-A^{-2}))*((-A^{2}-A^{-2})*(-A^{2}-A^{-2})^{2})]$
$=A^{-7}-A^{-3}-A^{5}$ $=\langle 3_{1} \rangle$.

\begin{remark}
For a polynomial ring $\mathbb{Z}[p,q,r]$, 
define a binary operation $*$ by $$a*b = pa+qb+r.$$ 
Then $(\mathbb{Z}[p,q,r], *)$ is entropic 
and it is easy to check that the Kauffman bracket magma in Example \ref{ExaKB} 
is obtained from the entropic magma $(\mathbb{Z}[p,q,r], *)$ by substituting $q= 1-p$.
Moreover, if we replace $p=A, q=A^{-1}$ and $r=0$ and take $a_{n}=(-A^{2}-A^{-2})^{n-1}$, 
then we obtain the Kauffman bracket magma in Example \ref{ExaKauffmanBracketInv}. 
If we take a non-zero sequence $\{a_{n}\}_{n \in \mathbb{N}} \subset \mathbb{Z}[p,q,r]$ for the triple $(\mathbb{Z}[p,q,r], *,\{a_{n}\})$ to be a Kauffman bracket magma, 
then it must be one of Kauffman magmas in Example \ref{ExaKB} or \ref{ExaKauffmanBracketInv}.
\end{remark}

For a Kauffman bracket magma $(A, *, \{a_{n}\})$, 
the invariant $P: {\mathcal L^{fr}} \rightarrow A$ is a regular isotopy invariant for links 
and then we can define new invariant $P^{\rho}$ to be invariant under Reidemeister move I 
by using $P$ with a specific map $\rho$.

Let ${\mathcal L}$ be the set of equivalence classes of links under all Reidemeister moves
and let $Map(A)$ be the set of maps from $A$ to itself.

\begin{theorem}\label{Thm0}
Assume that there is a map $\rho$ from ${\mathcal L^{fr}}$ to $Map(A)$ defined by 
for any link diagram $D$,
$$\rho(D)=\rho_{D}$$ such that 
$$\rho_{D}(P(D))=\rho_{D'}(P(D'))$$
where $D'$ is obtained from $D$ by applying once a positive or negative Reidemeister move I as depicted in Fig. \ref{RM1map}.

Then there exists unique invariant $P^{\rho} : {\mathcal L} \rightarrow A$ for links,
defined by $P^{\rho}(D)=\rho_{D}(P(D))$ for every link diagram $D$.
\begin{figure}[h!]
 \centering
 \includegraphics[width = 12cm]{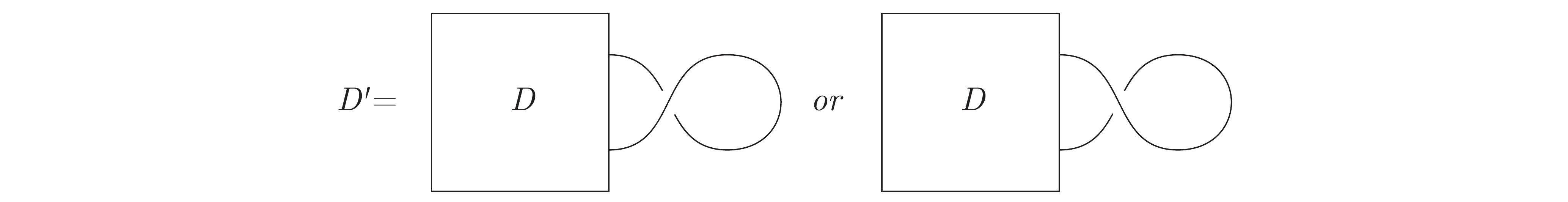}
 \caption{A positive or negative Reidemeister moves I}\label{RM1map}
\end{figure}
\end{theorem}

\begin{proof}
Construct a function $P^{\rho} : {\mathcal L} \rightarrow A$ 
defined by $P^{\rho}(D)=\rho_{D}(P(D))$ for every link diagram $D$
where
$${\mathcal L} \rightarrow Map(A)\times A \rightarrow A$$ defined by
$$D \mapsto (\rho_{D}, P(D)) \mapsto \rho_{D}(P(D)).$$
By hypothesis, $P^{\rho}$ is invariant under Reidemeister move I.
Since $\rho : {\mathcal L^{fr}} \rightarrow Map(A)$ and $P$ are regular isotopy invariants, 
$P^{\rho}$ is invariant under Reidemeister moves.
\end{proof}

\begin{example}
For the Kauffman bracket magma $(\mathbb{Z}[p,r]; *, \{a_{n}\})$ in Example \ref{ExaKB}, 
define a map $\rho : {\mathcal L^{fr}} \rightarrow Map(\mathbb{Z}[p,r])$ by for any link diagram $D$,
$\rho(D)=\rho_{D}$
where for any element $x\in \mathbb{Z}[p,r]$ 
\begin{equation*}
\rho_{D}(x) = \left\{
\begin{array}{cc} 
    x , & \text{if } c(D) \text{ is even}, \\
    -x+r ,& \text{if } c(D) \text{ is odd}, 
\end{array}\right.
\end{equation*}
and $c(D)$ is the number of crossings of $D$.
Since $\rho$ satisfies the condition of Theorem \ref{Thm0},
$P^{\rho} : \mathcal{L} \rightarrow \mathbb{Z}[p,r]$ is invariant for links.
Unfortunately, $P^{\rho}$ has only two values $1$ and $c$.  
\end{example}

\begin{example}
For the Kauffman bracket magma $(\mathbb{Z}[A^{\pm 1}]; *, \{a_{n}\})$ in Example \ref{ExaKauffmanBracketInv}, 
define a map $\rho : {\mathcal L^{fr}} \rightarrow Map(\mathbb{Z}[A^{\pm 1}])$ by for any oriented link diagram $D$,
$\rho(D)=\rho_{D}$
where 
$$\rho_{D}(x)=(-A^{3})^{-w(D)}x$$ for any element $x\in \mathbb{Z}[A^{\pm 1}]$ 
and $w(D)$ is the writhe of $D$.
Since $\rho$ satisfies the condition of Theorem \ref{Thm0},
$P^{\rho} : \mathcal{L} \rightarrow \mathbb{Z}[A^{\pm 1}]$ is invariant for oriented links.
In fact, the invariant $P^{\rho}(D)$ is equal to the normalized bracket polynomial $X(D)$.
\end{example}



\section{Marked Kauffman bracket magma and its invariant}
\subsection{Marked Kauffman bracket magmas}

In order to construct invariants for surface-links in $4$-space, 
we can generalize a Kauffman bracket magma 
and construct invariants for marked graph diagrams under unoriented Yoshikawa moves in Fig. \ref{UnoriYoshikawaMoves}

\begin{definition}\label{DefMKBM}
Let $(A;*,\{a_{n}\})$ be a Kauffman bracket magma and $\bullet$ a binary operation on $A$.
If $(A; \bullet)$ is entropic,
then the quadruple $(A;*,\bullet,\{a_{n}\})$ is called a {\it marked Kauffman bracket magma}.
\end{definition}

The two binary operation $*$ and $\bullet$ are corresponding to the Kauffman bracket skein relation and the marked skein relation, respectively.

\begin{figure}[h!]
 \centering
 \includegraphics[width = 12cm]{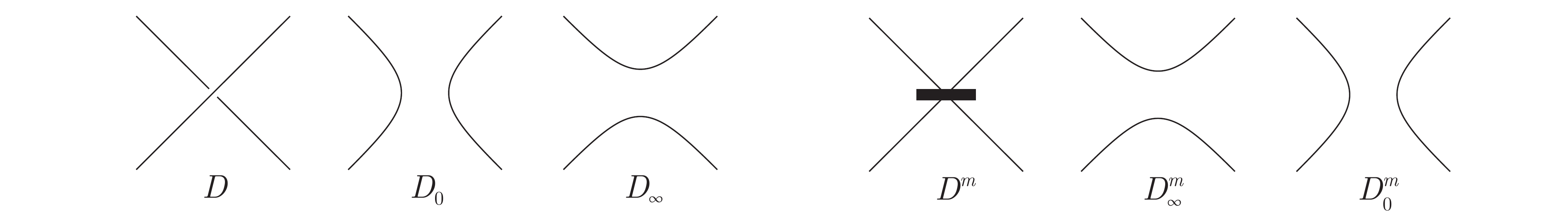}
 \caption{The Kauffman skein triple and the marked skein triple}\label{SkeinMarkedSkein}
\end{figure}

Now we construct the invariant $P_{M}$ from the set $\mathcal{F}$ of admissible marked graph diagrams to a marked Kauffman bracket magma $(A;*,\bullet,\{a_{n}\})$.

If the diagram $D$ has markers, then for each marker $m$, we splice $m$ satisfying
\begin{equation}\label{MarkedSkeinRel}
P_{M}(D^{m}) = P_{M}(D^{m}_{\infty}) \bullet P_{M}(D^{m}_{0}).
\end{equation}
where $(D^{m}, D^{m}_{\infty}, D^{m}_{0})$ is the marked skein triple depicted in Fig. \ref{SkeinMarkedSkein}.

If the diagram $D$ has no markers, then 
\begin{equation}
P_{M}(D) = P(D),
\end{equation}
where $P$ is the invariant valued in the Kauffman bracket magma $(A; *, \{a_{n}\})$.

\begin{lemma}\label{LemWD}
The function $P_{M} : \mathcal{F} \rightarrow A$ is well-defined.
\end{lemma}

\begin{proof}
First we will show that $P_{M}$ is well-defined for marked graph diagrams with ordered markers. 
Let $D$ be a marked graph diagram with $n$ ordered markers $\{m_{1}, \dots, m_{n}\}$. 
If $n=0$, then $P_{M}(D)=P(D)$ and it is well-defined. 
Assume that $P_{M}$ is well-defined for marked graph diagrams with $k$ ordered markers 
for some $k \in \mathbb{N}\cup \{ 0 \}$. 
Let $D$ be a marked graph diagrams with $(k+1)$ ordered markers $\{m_{1}, \dots, m_{k+1}\}$. 
By Equation \ref{MarkedSkeinRel} at the marker $m_{k+1}$,
$P_{M}(D^{m_{k+1}}) = P_{M}(D_{\infty}^{m_{k+1}}) \bullet P_{M}(D_{0}^{m_{k+1}})$. 
Since $D_{\infty}^{m_{k+1}}$ and $D_{0}^{m_{k+1}}$ have $k$ ordered markers, 
by the hypothesis,
the values $P_{M}(D_{\infty}^{m_{k+1}})$ and $P_{M}(D_{0}^{m_{k+1}})$ are well-defined 
and hence $P_{M}(D^{m_{k+1}})$ is well-defined by Mathematical Induction.

Second we will prove that $P_{M}$ does not depend on the order of markers. 
Let $m_{1}$ and $m_{2}$ be two markers of a marked graph diagram $D$. 
It is sufficient to show that two values corresponding to the ordered markers $\{m_{1}, m_{2}\}$ and the ordered markers $\{m_{2}, m_{1}\}$ are the same.
The marked graph diagram is obtained from $D$ by splicing $m_{1}$ in a state $\alpha$ and $m_{2}$ in a state $\beta$ for $\alpha, \beta \in \{0, \infty \}$ and denoted by $D^{m_{1},m_{2}}_{\alpha,\beta}$.
In the ordered markers $\{m_{1}, m_{2}\}$,
\begin{eqnarray*}
P_{M}(D_{*,*}^{m_{1},m_{2}}) &=& P_{M}(D^{m_{1},m_{2}}_{\infty, *}) \bullet P_{M}(D^{m_{1},m_{2}}_{0, *})\\
&=&[P_{M}(D^{m_{1},m_{2}}_{\infty, \infty}) \bullet P_{M}(D^{m_{1},m_{2}}_{\infty, 0})] \bullet [ P_{M}(D^{m_{1},m_{2}}_{0, \infty}) \bullet P_{M}(D^{m_{1},m_{2}}_{0, 0})],
\end{eqnarray*}
and in the ordered markers $\{m_{2}, m_{1}\}$,
\begin{eqnarray*}
P_{M}(D_{*,*}^{m_{1},m_{2}}) &=& P_{M}(D^{m_{1},m_{2}}_{*,\infty})\bullet P_{M}(D^{m_{1},m_{2}}_{*,0})\\
&=&[P_{M}(D^{m_{1},m_{2}}_{\infty, \infty}) \bullet P_{M}(D^{m_{1},m_{2}}_{0, \infty})]\bullet [ P_{M}(D^{m_{1},m_{2}}_{\infty,0}) \bullet P_{M}(D^{m_{1},m_{2}}_{0, 0})],
\end{eqnarray*}
Since $(A,\bullet)$ is an entropic, one can see that 
\begin{eqnarray*}
&[P_{M}(D^{m_{1},m_{2}}_{\infty, \infty}) \bullet P_{M}(D^{m_{1},m_{2}}_{\infty, 0})]\bullet [P_{M}(D^{m_{1},m_{2}}_{0, \infty}) \bullet P_{M}(D^{m_{1},m_{2}}_{0, 0})]
\\ =& 
[P_{M}(D^{m_{1},m_{2}}_{\infty, \infty}) \bullet P_{M}(D^{m_{1},m_{2}}_{0, \infty})]\bullet [ P_{M}(D^{m_{1},m_{2}}_{\infty,0}) \bullet P_{M}(D^{m_{1},m_{2}}_{0, 0})].
\end{eqnarray*}
Therefore $P_{M}(D^{m_{1},m_{2}})$ does not depend on the order of markers and the well-definedness of $P_{M}$ is proved.
\end{proof}

\begin{lemma}\label{LemGamma2345}
For a marked graph diagram $D$, $P_{M}$ is invariant under $\Gamma_{2}$, $\Gamma_{3}$, $\Gamma_{4}$, $\Gamma'_{4}$ and $\Gamma_{5}$.
\end{lemma}

\begin{proof}
Let $D$ and $D'$ be two marked graph diagrams 
such that one of them are obtained from the other by only one of Yoshikawa moves $\Gamma_2, \Gamma_3$, $\Gamma_4, \Gamma'_4$ and $\Gamma_5$.
Since the value $P_{M}$ does not depend on the order of splicings of markers,
all markers of $D$ and $D'$ except for markers contained in the applied move are spliced.
It is sufficient to consider $D$ and $D'$ with no markers except for markers contained in the applied moves.

For the moves $\Gamma_2$ and $\Gamma_{3}$, 
since there are no markers in $\Gamma_{2}$ and $\Gamma_{3}$, that is,  
$$P_{M}(D) = P(D) = P(D') = P_{M}(D'),$$
it is obvious that $P_{M}$ is invariant under $\Gamma_{2}$ and $\Gamma_{3}$.

For the move $\Gamma_{4}$, 
let $D$ and $D'$ be two marked graph diagrams obtained from each other by $\Gamma_{4}$ 
and let $m_{1}$ and $m_{2}$ denote two marked vertices of $\Gamma_{4}$, 
as described in Fig. \ref{InvUnoriGamma4}. 
\begin{figure}[h!]
 \centering
 \includegraphics[width = 12cm]{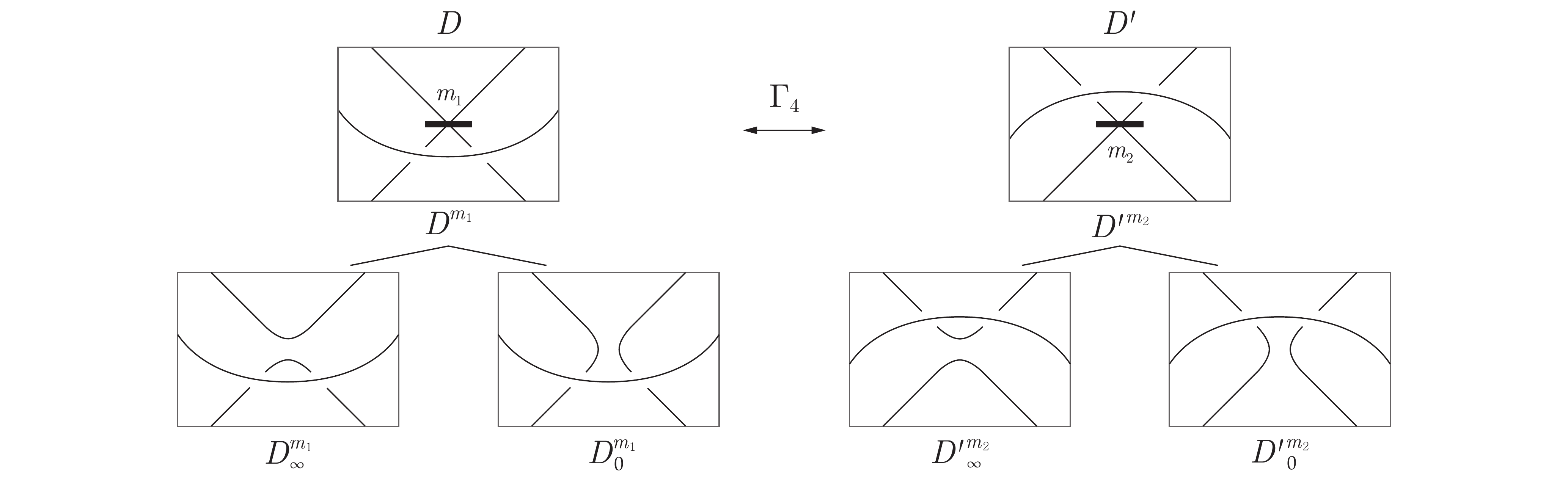}
 \caption{The resolving tree of $\Gamma_{4}$}\label{InvUnoriGamma4}
\end{figure}
Then
$$P_{M}(D)=P_{M}(D^{m_{1}})=P_{M}(D^{m_{1}}_{\infty}) \bullet P_{M}(D^{m_{1}}_{0})$$ 
and 
$$P_{M}(D')=P_{M}(D'^{m_{2}})=P_{M}(D'^{m_{2}}_{\infty}) \bullet P_{M}(D'^{m_{2}}_{0}).$$
Note that $D^{m_{1}}_{0}$ and $D'^{m_{2}}_{0}$ are the same diagrams 
and $D^{m_{1}}_{\infty}$ and $D'^{m_{2}}_{\infty}$ are equivalent under $\Gamma_{2}$. 
Since $P_{M}$ is invariant under $\Gamma_{2}$, 
$P_{M}(D^{m_{1}}_{\infty}) = P_{M}(D'^{m_{2}}_{\infty})$ and $P_{M}(D^{m_{1}}_{0})=P_{M}(D'^{m_{2}}_{0})$ 
and hence $P_{M}(D) = P_{M}(D')$. 
Analogously, $P_{M}$ is invariant under $\Gamma_{4}'$.

For the move $\Gamma_{5}$, 
let $D$ and $D'$ be two marked graph diagrams obtained from each other by $\Gamma_{5}$ 
and let $c_{1}, c_{2}$ and $m_{1}, m_{2}$ denote two classical crossings and two marked vertices of $\Gamma_{5}$, respectively, 
as described in Fig. \ref{InvUnoriGamma5}. 
\begin{figure}[h!]
 \centering
 \includegraphics[width = 12cm]{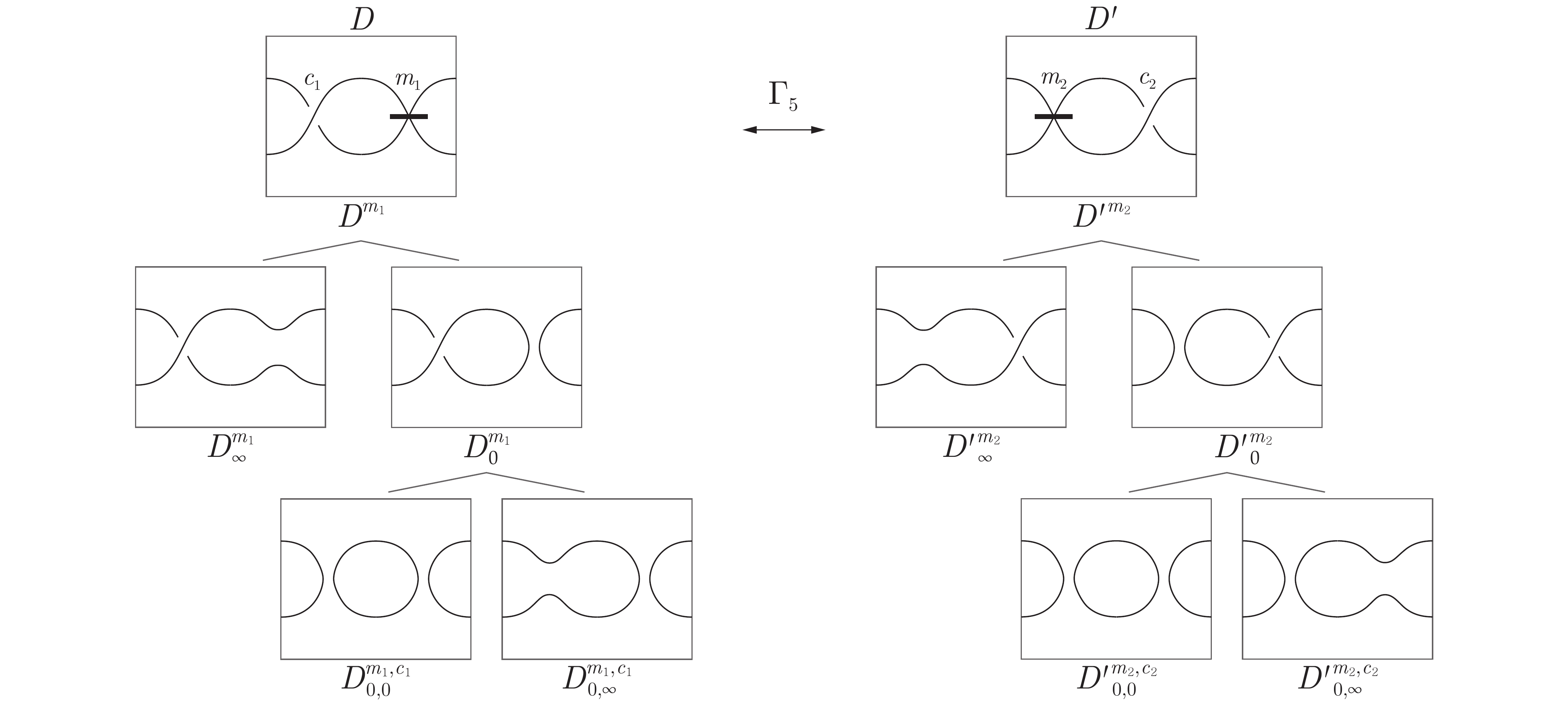}
 \caption{The resolving tree of $\Gamma_{5}$}\label{InvUnoriGamma5}
\end{figure}
Then their calculation is as follows :
$$P_{M}(D)=P_{M}(D^{m_{1}}_{\infty}) \bullet P_{M}(D^{m_{1}}_{0})=P_{M}(D^{m_{1}}_{\infty}) \bullet (P_{M}(D^{m_{1}, c_{1}}_{0, 0})*P_{M}(D^{m_{1}, c_{1}}_{0, \infty}))$$  
and 
$$P_{M}(D')=P_{M}(D'^{m_{2}}_{\infty}) \bullet P_{M}(D'^{m_{2}}_{0})=P_{M}(D'^{m_{2}}_{\infty}) \bullet (P_{M}(D'^{m_{2}, c_{2}}_{0, 0})*P_{M}(D'^{m_{2}, c_{2}}_{0, \infty})).$$ 
Note that the above second equalities are obtained by applying the Kauffman skein triple at classical crossings $c_{1}$ and $c_{2}$.
Since 
$D^{m_{1}}_{\infty}$ and $D'^{m_{2}}_{\infty}$ are the same diagrams
and since
$D^{m_{1}, c_{1}}_{0, 0}$ and $D^{m_{1}, c_{1}}_{0, \infty}$ are the same as 
$D'^{m_{2}, c_{2}}_{0, 0}$ and $D'^{m_{2}, c_{2}}_{0, \infty}$, respectively,
$P_{M}(D^{m_{1}}_{\infty})=P_{M}(D'^{m_{2}}_{\infty})$ and $P_{M}(D^{m_{1}}_{0})=P_{M}(D'^{m_{2}}_{0})$.
Hence $P_{M}(D)=P_{M}(D')$.
\end{proof}

\begin{remark}
In the proof of Lemma \ref{LemWD}, 
by using the entropic condition $(a\bullet b) \bullet (c\bullet d) = (a \bullet c) \bullet (b \bullet d)$, 
we can showy that $P_{M}$ is independent on the order of markers. 
In the same manner, the entropic property between two binary operations $*$ and $\bullet$, 
$$(a* b) \bullet (c * d) = (a \bullet c) * (b \bullet d),$$
means that the value $P_{M}$ does not depend on the order of splicings of a classical crossing and a marked vertex. 
\end{remark}

\begin{remark}
For the proof of Lemma \ref{LemGamma7and8}, 
consider the elementary $3$- and $4$-tangle diagrams,
where the {\it elementary $n$-tangle diagram} means that the $n$-tangle diagram has no crossings and no circle components.
Then it is easy to check that there exist 
$5$ types of the elementary $3$-tangle diagrams and $14$ types of the elementary $4$-tangle diagrams 
illustrated in Fig. \ref{Elem3tangle} and \ref{Elem4tangle}, respectively.
\begin{figure}[h!]
 \centering
 \includegraphics[width = 12cm]{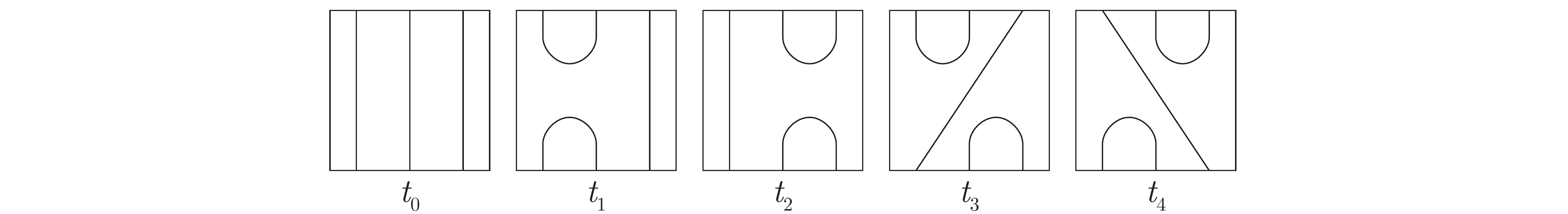}
 \caption{Elementary $3$-tangle diagrams}\label{Elem3tangle}
\end{figure}

\begin{figure}[h!]
 \centering
 \includegraphics[width = 12cm]{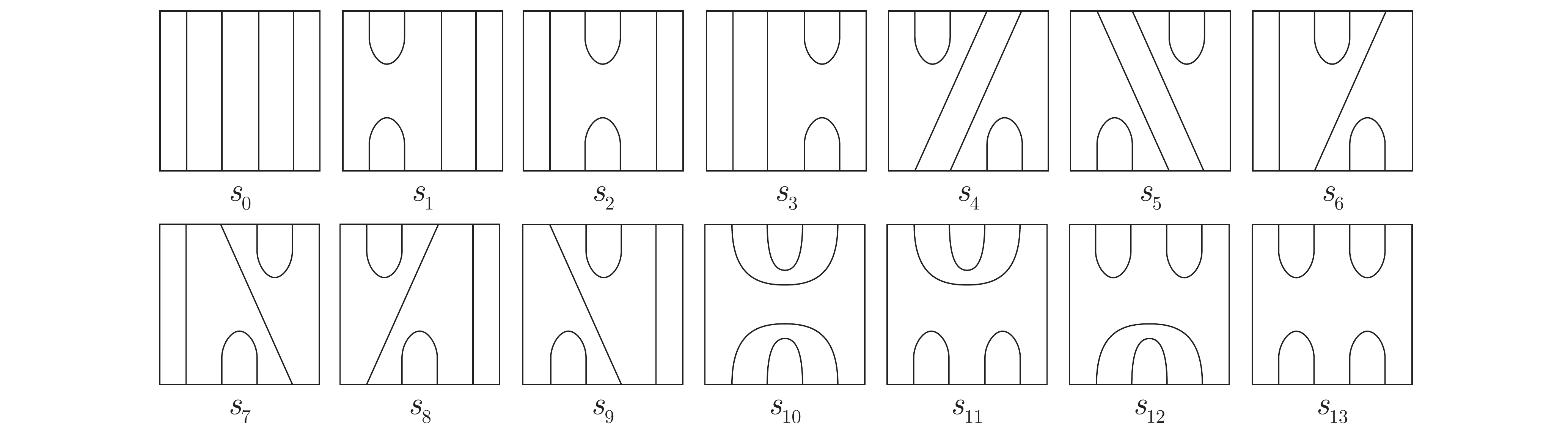}
 \caption{Elementary $4$-tangle diagrams}\label{Elem4tangle}
\end{figure}
\end{remark}

\begin{lemma}\label{LemGamma7and8}
Let $(A;*,\bullet,\{a_{n}\})$ be a marked Kauffman bracket magma satisfying $(a* b) \bullet (c * d) = (a \bullet c) * (b \bullet d)$, for all $a, b, c, d \in A$.
If $a_{n}= a_{n+2}$ for all $n\in \Neal$,
then $P_{M}$ is invariant under Yoshikawa moves $\Gamma_{7}$ and $\Gamma_{8}$.
\end{lemma}

\begin{proof}
In order to prove that $P_{M} : \mathcal{F} \rightarrow A$ is an invariant under $\Gamma_{7}$,
let $D$ and $D'$ be two marked graph diagrams where $D'$ is obtained from $D$ by applying only once $\Gamma_{7}$.
Let $T_{7}$ and $T_{7}'$ denote the $3$-tangle diagrams corresponding to $\Gamma_{7}$ 
where $\{m_1, m_2\}$ and $\{m'_1, m'_2\}$ are the sets of markers of $D$ and $D'$ respectively, as depicted in Fig. \ref{Gamma7}.
\begin{figure}[h!]
 \centering
 \includegraphics[width = 12cm]{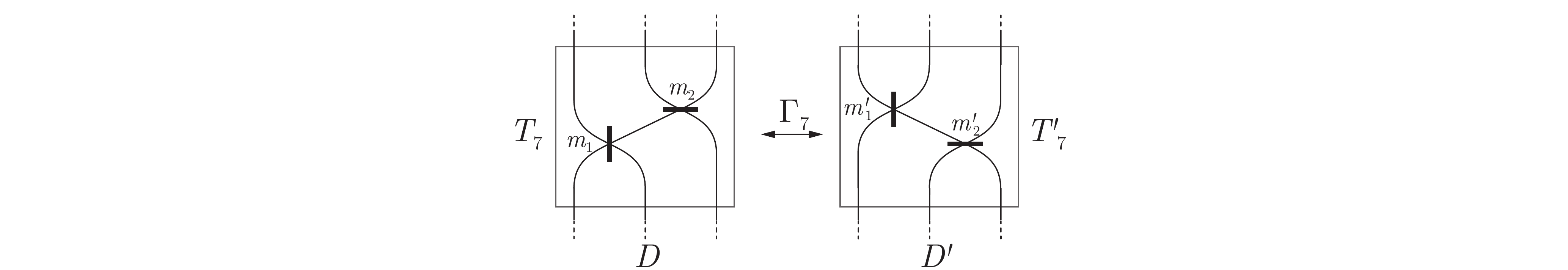}
 \caption{Yoshikawa move $\Gamma_7$}\label{Gamma7}
\end{figure}
If $D$ and $D'$ have other markers in outside of $T_{7}$ and $T'_{7}$, 
by splicing all of them,
then the result diagrams obtained from $D$ and $D'$ has no markers except $\{m_1, m_2\}$ and $\{m'_1, m'_2\}$, respectively.
One can assume that $D$ and $D'$ have no markers except $\{m_1, m_2\}$ and $\{m'_1, m'_2\}$ in $\Gamma_{7}$.

We can deform $D$ into the closure $cl(T_{7}S)$ of the sum of the $3$-tangle diagram $T_{7}$ and a $3$-tangle diagram $S$ as shown in Fig. \ref{Gamma7-2}, 
where $S$ is a $3$-tangle diagram whose all crossings are classical. 
\begin{figure}[h!]
 \centering
 \includegraphics[width = 12cm]{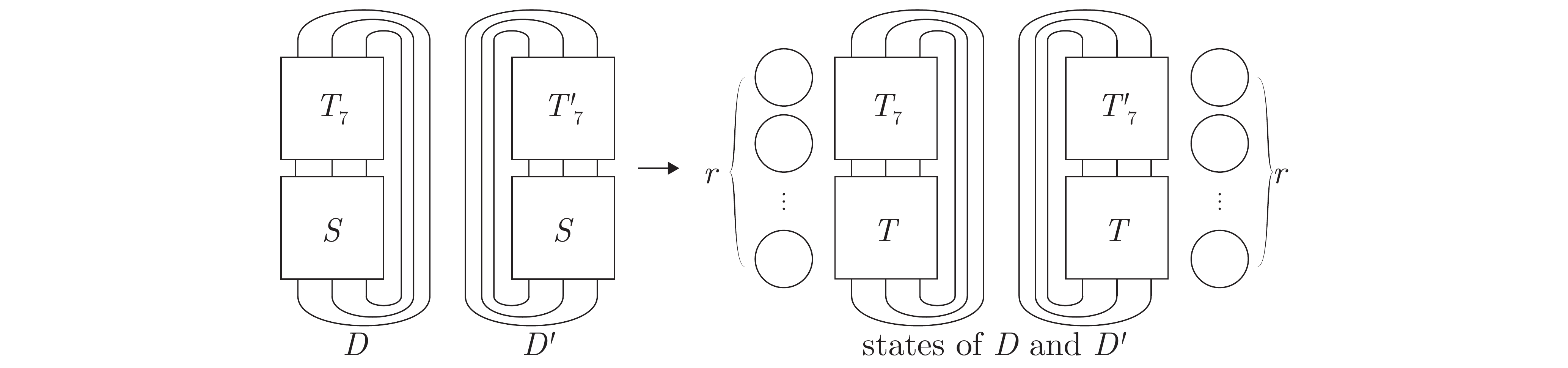}
 \caption{The states of $cl(T_{7}S)$ and $cl(T'_{7}S)$}\label{Gamma7-2}
\end{figure}
By the relation $(a* b) \bullet (c * d) = (a \bullet c) * (b \bullet d)$ and by applying the Kauffman skein triple at all crossings of $S$, 
the result states have only two marked vertices $m_1$ and $m_2$.
That is, as described in Fig. \ref{Gamma7-2},
each states can be expressed as the disjoint union $O_{r}\sqcup cl(T_{7}T)$ 
where $O_{r}$ is the trivial link diagram with $r$ components
and $T$ is one of the elementary $3$-tangle diagrams $t_0$, $t_1$, $\cdots$, $t_4$ in Fig. \ref{Elem3tangle}.
Then the value $P_{M}(D)=P_{M}(cl(T_{7}S))$ can be expressed as the value based on $P_{M}(O_{r}\sqcup cl(T_{7}T))$ by applying operations $*$ and $\bullet$.

In the same manner, 
the value $P_{M}(D')=P_{M}(cl(T'_{7}S))$ can be expressed as the value based on $P_{M}(O_{r}\sqcup cl(T'_{7}T))$ by applying operations $*$ and $\bullet$. 
Therefore, if $P_{M}(O_{r}\sqcup cl(T_{7}T))$ and $P_{M}(O_{r}\sqcup cl(T'_{7}T))$ are the same for all $T=t_0$, $t_1$, $\cdots$, $t_4$, then $P_{M}(D)=P_{M}(D')$.
By calculating the invariant $P_{M}(cl(T_{7}T))$ and $P_{M}(cl(T'_{7}T))$ valued in $A$ shown in Fig. \ref{Gamma7-1},
one can obtain two equations
$$P_{M}(cl(T_{7}T))= \left(P_{M}(cl(t_{2}T))\bullet P_{M}(cl(t_{0}T))\right)\bullet \left(P_{M}(cl(t_{4}T))\bullet P_{M}(cl(t_{1}T))\right)$$
and
$$P_{M}(cl(T'_{7}T))= \left(P_{M}(cl(t_{2}T))\bullet P_{M}(cl(t_{0}T))\right)\bullet \left(P_{M}(cl(t_{3}T))\bullet P_{M}(cl(t_{1}T))\right).$$
Then two equations are the same except the values $P_{M}(cl(t_{4}T))$ and $P_{M}(cl(t_{3}T))$.

\begin{figure}[h!]
 \centering
 \includegraphics[width = 12cm]{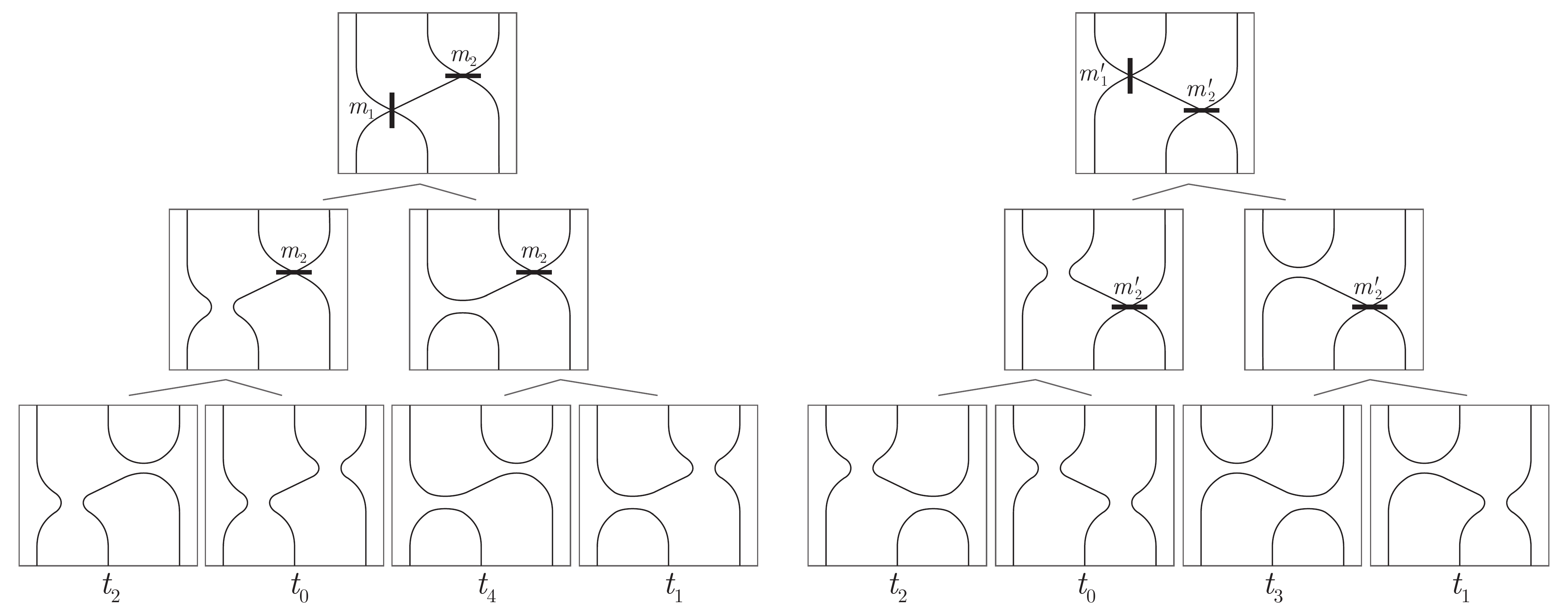}
 \caption{The resolving tree of $\Gamma_{7}$}\label{Gamma7-1}
\end{figure}

For each $T=t_0$, $t_1$, $\cdots$, $t_4$, 
$P_{M}(cl(t_{4}T))$ and $P_{M}(cl(t_{3}T))$ are as follows.
$$
\begin{array}{c|ccccc}
  T & t_0 & t_1 & t_2 & t_3 & t_4   \\
  \hline
  P_{M}(cl(t_{4}T)) & a_1 & a_2 & a_2 & a_3 & a_1  \\
  P_{M}(cl(t_{3}T)) & a_1 & a_2 & a_2 & a_1 & a_3 
\end{array}
$$
If $a_{1}= a_{3}$, then $P_{M}(cl(t_{4}T))=P_{M}(cl(t_{3}T))$.
By considering the trivial link diagram $O_{r}$,
since $a_{n}= a_{n+2}$ for all $n\in \Neal$, 
$P_{M}(O_{r}\sqcup cl(T_{7}T))$ and $P_{M}(O_{r}\sqcup cl(T'_{7}T))$.
Hence if $a_{n}= a_{n+2}$ for all $n\in \Neal$,
then $P_{M} : \mathcal{F} \rightarrow A$ is invariant under $\Gamma_{7}$.

Secondly, we will show that $P_{M} : \mathcal{F} \rightarrow A$ is an invariant under $\Gamma_{8}$.
Similarly to $\Gamma_{7}$,
let $D$ and $D'$ be two marked graph diagrams where $D'$ is obtained from $D$ by applying only once $\Gamma_{8}$.
Let $T_{8}$ and $T_{8}'$ denote the $4$-tangle diagrams corresponding to $\Gamma_{8}$ 
where $\{m_1, m_2\}$ and $\{m'_1, m'_2\}$ are the sets of markers of $D$ and $D'$ respectively, as depicted in Fig. \ref{Gamma8}.

\begin{figure}[h!]
 \centering
 \includegraphics[width = 12cm]{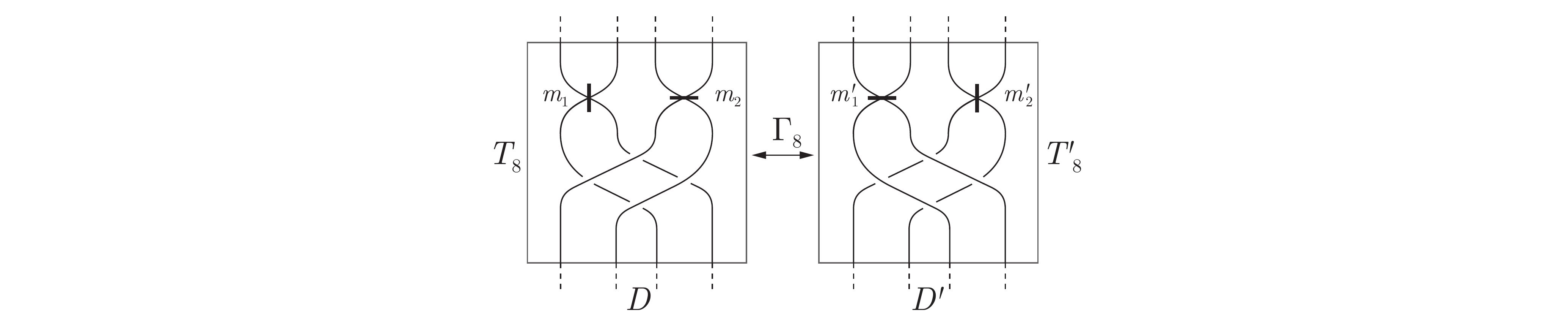}
 \caption{Yoshikawa move $\Gamma_8$}\label{Gamma8}
\end{figure}

As the proof for $\Gamma_{7}$,
one can assume that both $D$ and $D'$ have no markers except markers in $\Gamma_{8}$ 
and
$D$ can be deformed to the closure of the sum of the $4$-tangle diagram $T_{8}$ and a $4$-tangle diagram $S$ as shown in Fig. \ref{Gamma8-2}.
By applying the Kauffman skein triple at all crossings of $S$,
one can obtain the result diagrams as the disjoint union $O_{r}\sqcup cl(T_{8}T)$ of the trivial link diagram $O_{r}$ with $r$ components and the closures of $T_{8}T$, as shown in Fig. \ref{Gamma8-2}, where $T$ is one of the elementary $4$-tangle diagrams $s_0$, $s_1$, $\cdots$, $s_{13}$ in Fig. \ref{Elem4tangle}.
\begin{figure}[h!]
 \centering
 \includegraphics[width = 12cm]{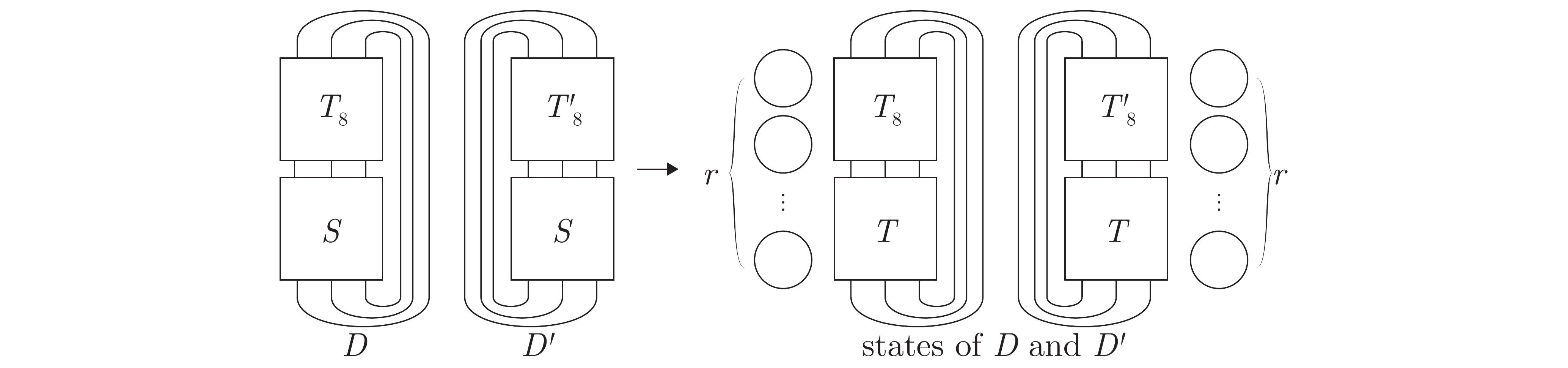}
 \caption{The states of $cl(T_{8}S)$ and $cl(T'_{8}S)$}\label{Gamma8-2}
\end{figure}
Then the value $P_{M}(D)=P_{M}(cl(T_{8}S))$ can be expressed as the value based on $P_{M}(O_{r}\sqcup cl(T_{8}T))$ by applying operations $*$ and $\bullet$.
In the same manner, the value $P_{M}(D')=P_{M}(cl(T'_{8}S))$ can be expressed as the value based on $P_{M}(O_{r}\sqcup cl(T'_{8}T))$ by applying operations $*$ and $\bullet$.
Therefore, if $P_{M}(O_{r}\sqcup cl(T_{8}T))$ and $P_{M}(O_{r}\sqcup cl(T'_{8}T))$ are the same for each $T=s_0$, $s_1$, $\cdots$, $s_{13}$, then $P_{M}(D)=P_{M}(D')$.

By calculating the invariant $P_{M}(cl(T_{8}T))$ and $P_{M}(cl(T'_{8}T))$ valued in $A$ as depicted in Fig. \ref{Gamma8-1},
\begin{figure}[h!]
 \centering
 \includegraphics[width = 12cm]{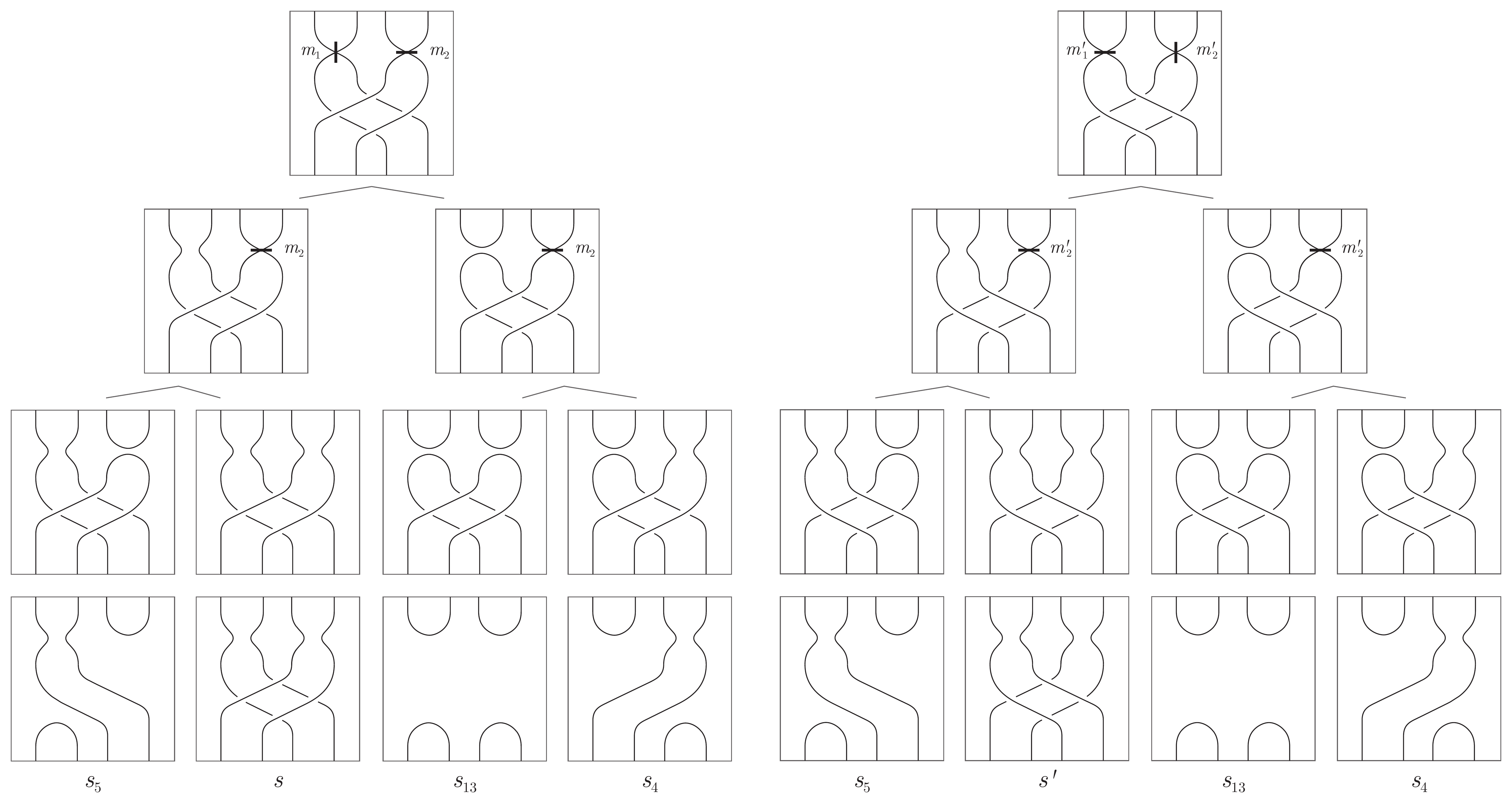}
 \caption{The resolving tree of $\Gamma_{8}$}\label{Gamma8-1}
\end{figure}
one can obtain two equations
$$P_{M}(cl(T_{8}T))= \left(P_{M}(cl(s_{5}T))\bullet P_{M}(cl(sT))\right)\bullet \left(P_{M}(cl(s_{13}T))\bullet P_{M}(cl(s_{4}T))\right)$$
and
$$P_{M}(cl(T'_{8}T))= \left(P_{M}(cl(s_{5}T))\bullet P_{M}(cl(s'T))\right)\bullet \left(P_{M}(cl(s_{13}T))\bullet P_{M}(cl(s_{4}T))\right).$$
Two equations are different between $P_{M}(cl(sT))$ and $P_{M}(cl(s'T))$.
Then the values $P_{M}(cl(sT))$ and $P_{M}(cl(s'T))$ are as follows ;
$$
\begin{array}{c|cccccccccccccc}
  T                     & s_0 & s_1 & s_2 & s_3 & s_4 & s_5 & s_6 & s_7 & s_8 & s_9 & s_{10} & s_{11} & s_{12} & s_{13}  \\
  \hline
  P_{M}(cl(sT)) & x_0  & a_1 & x_2 & a_1 & a_3 & a_3 & a_2 & a_2 & a_2 & a_2 & x_{10} & a_1 & a_1 & a_2 \\
  P_{M}(cl(s'T)) & y_0 & a_1 & y_2 & a_1 & a_3 & a_3 & a_2 & a_2 & a_2 & a_2 & y_{10} & a_1 & a_1 & a_2
\end{array}
$$
where \\
{\footnotesize
$x_{0}=$
$\left( \left( (a_{2}*a_{1})*(a_{3}*a_{2}) \right) * \left( (a_{3}*a_{2})*(a_{4}*a_{3}) \right) \right) 
* \left( \left( (a_{1}*a_{2})*(a_{2}*a_{3}) \right) * \left( (a_{2}*a_{3})*(a_{3}*a_{4}) \right) \right) $
$y_{0}=$
$\left( \left( (a_{4}*a_{3})*(a_{3}*a_{2}) \right) * \left( (a_{3}*a_{2})*(a_{2}*a_{1}) \right) \right) 
* \left( \left( (a_{3}*a_{4})*(a_{2}*a_{3}) \right) * \left( (a_{2}*a_{3})*(a_{1}*a_{2}) \right) \right) $
$x_{2}=$
$\left( \left( (a_{3}*a_{2})*(a_{4}*a_{3}) \right) * \left( (a_{4}*a_{3})*(a_{5}*a_{4}) \right) \right) 
* \left( \left( (a_{2}*a_{1})*(a_{3}*a_{2}) \right) * \left( (a_{3}*a_{2})*(a_{4}*a_{3}) \right) \right) $
$y_{2}=$
$\left( \left( (a_{3}*a_{4})*(a_{2}*a_{3}) \right) * \left( (a_{2}*a_{3})*(a_{1}*a_{2}) \right) \right) 
* \left( \left( (a_{4}*a_{5})*(a_{3}*a_{4}) \right) * \left( (a_{3}*a_{4})*(a_{2}*a_{3}) \right) \right) $
$x_{10}=$
$\left( \left( (a_{4}*a_{3})*(a_{3}*a_{2}) \right) * \left( (a_{3}*a_{2})*(a_{4}*a_{3}) \right) \right) 
* \left( \left( (a_{3}*a_{2})*(a_{2}*a_{1}) \right) * \left( (a_{2}*a_{1})*(a_{3}*a_{2}) \right) \right) $
$y_{10}=$
$\left( \left( (a_{2}*a_{3})*(a_{1}*a_{2}) \right) * \left( (a_{1}*a_{2})*(a_{2}*a_{3}) \right) \right) 
* \left( \left( (a_{3}*a_{4})*(a_{2}*a_{3}) \right) * \left( (a_{2}*a_{3})*(a_{3}*a_{4}) \right) \right) $.
}
The values $P_{M}(cl(sT))$ and $P_{M}(cl(s'T))$ are the same except $T= s_{0}, s_{2}, s_{10}$.
Since $a_{n}=a_{n+2}$, we obtain $x_{0}=y_{0}$, $x_{2}=y_{2}$ and $x_{10}=y_{10}$.
Then $P_{M}(cl(sT))=P_{M}(cl(s'T))$ and hence $P_{M}(O_{r}\sqcup cl(T_{8}T))$ and $P_{M}(O_{r}\sqcup cl(T'_{8}T))$.
Therefore $P_{M} : \mathcal{F} \rightarrow A$ is invariant under $\Gamma_{8}$.
\end{proof}

\begin{theorem}\label{Thm1}
Let $(A;*,\bullet,\{a_{n}\})$ be a marked Kauffman bracket magma satisfying the following conditions : 
\begin{enumerate}
\item $a_{n}= a_{n+2}$, for all $n\in \Neal$,
\item $(a* b) \bullet (c * d) = (a \bullet c) * (b \bullet d)$, for all $a,b,c,d \in A$.
\end{enumerate}
Then $P_{M} : \mathcal{F} \rightarrow A$ is invariant under Yoshikawa moves except $\Gamma_{1}, \Gamma_{6}$ and $\Gamma_{6}'$.
\end{theorem}

\begin{proof}
By the construction of $P_{M}$ and Lemma \ref{LemWD}, \ref{LemGamma2345} and \ref{LemGamma7and8}, this theorem holds. 
\end{proof}

\subsection{Kauffman bracket magma and Yoshikawa moves $\Gamma_{1}, \Gamma_{6}$ and $\Gamma_{6}'$}

In this section, we discuss on the invariance under $\Gamma_{1}$, $\Gamma_{6}$ and $\Gamma_{6}'$.

\begin{lemma}\label{LemGamma6}
Let $(A;*,\bullet,\{a_{n}\})$ be a marked Kauffman bracket magma satisfying 
$(a* b) \bullet (c * d) = (a \bullet c) * (b \bullet d)$ for all $a, b, c, d \in A$.
If $a_{n+1}\bullet a_{n}=a_{n}= a_{n}\bullet a_{n+1}$ for all $n\in \Neal$,
then $P_{M}$ is invariant under Yoshikawa moves $\Gamma_{6}$ and $\Gamma_{6}'$.
\end{lemma}

\begin{proof}
Let $D$ be an oriented marked graph diagram and let $D'$ be an oriented marked graph diagram obtained from $D$ by applying only once Yoshikawa move $\Gamma_{6}$.
Then the diagrams $D$ and $D'$ are the same except the part of $\Gamma_{6}$.
By the construction of $P_{M}$, 
without loss of generality, 
we can assume that there are no markers outside of the part of $\Gamma_{6}$.
That is, $D$ has only one marker, say $m$, and $D'$ has no markers as depicted in Fig. \ref{Gamma6pf(unori)1}.

\begin{figure}[h!]
 \centering
 \includegraphics[width = 12cm]{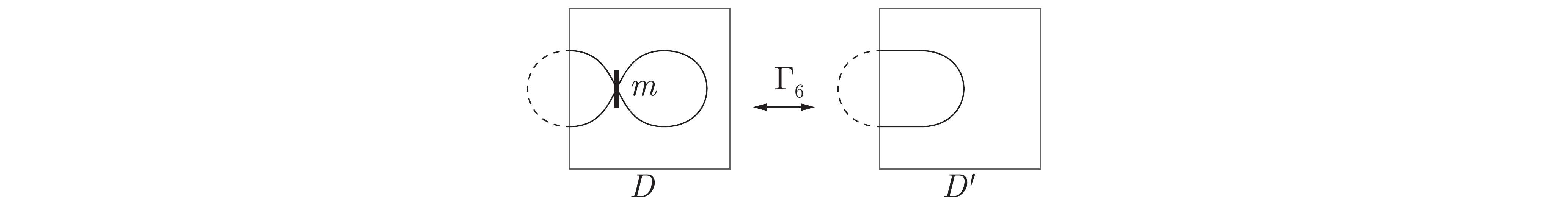}
 \caption{Yoshikawa move $\Gamma_{6}$} \label{Gamma6pf(unori)1}
\end{figure}

From the relation $(a* b) \bullet (c * d) = (a \bullet c) * (b \bullet d)$, we can splice at all classical crossings of $D$ and $D'$.
Then each state $sD$ of $D$ is expressed as the disjoint union of the $r$-component trivial link diagram and the diagram which has one marker $m$ and has no classical crossings and one $sD'$ of $D'$ is the $(r+1)$-component trivial link diagram illustrated in Fig. \ref{Gamma6pf(unori)2}.

\begin{figure}[h!]
 \centering
 \includegraphics[width = 12cm]{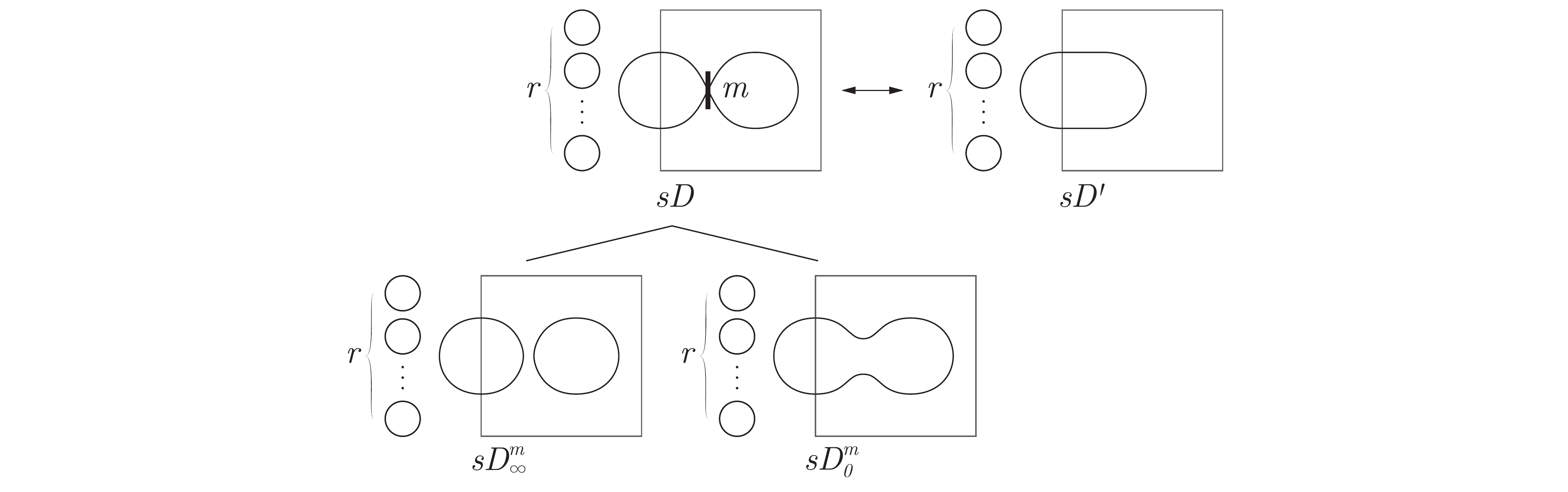}
 \caption{The resolving tree of $\Gamma_{6}$} \label{Gamma6pf(unori)2}
\end{figure}

By splicing at the marker $m$, the calculation of $sD$ and $sD'$ are as follows :
$$P_{M}(sD)=P_{M}(sD^{m})=P_{M}(sD^{m}_{\infty})\bullet P_{M}(sD^{m}_{0})=P(sD^{m}_{\infty})\bullet P(sD^{m}_{0})=a_{r+2}\bullet a_{r+1}$$
and 
$$P_{M}(sD')=P(sD')=a_{r+1}.$$
Since $a_{n+1}\bullet a_{n}=a_{n}$ for any $n\in \mathbb{N}$, $P_{M}(sD) = a_{r+1} \bullet a_{r} = a_{r} = P_{M}(sD')$
and then $P_{M}(D) = P_{M}(D')$. Hence $P_{M}$ is invariant under $\Gamma_{6}$.
In the same manner, $P_{M}$ is invariant under $\Gamma_{6}'$ by the condition $a_{n}\bullet a_{n+1}=a_{n}$.
\end{proof}

From Theorem \ref{Thm1} and Lemma \ref{LemGamma6}, the following theorem holds.

\begin{theorem}\label{Thm2}
Let $\mathcal{F}$ be the set of adimissible marked graph diagrams.
For a marked Kauffman bracket magma $(A;*,\bullet,\{a_{n}\})$ satisfying three conditions : 
\begin{enumerate}
  \item $a_{n}= a_{n+2}$, for all $n\in \Neal$,
  \item $a_{n+1}\bullet a_{n}=a_{n}= a_{n}\bullet a_{n+1}$, for all $n\in \Neal$,
  \item $(a* b) \bullet (c * d) = (a \bullet c) * (b \bullet d)$, for all $a,b,c,d \in A$.
\end{enumerate}
Then $P_{M} : \mathcal{F} \rightarrow A$ is invariant under Yoshikawa moves except $\Gamma_{1}$.
\end{theorem}

For a marked Kauffman bracket magma $(A; *, \bullet, \{a_{n}\})$, 
consider the invariant $P_{M}: {\mathcal F} \rightarrow A$ for admissible marked graph diagrams 
under Yoshikawa moves except $\Gamma_{1}$.
Let $Map(A)$ be the set of maps from $A$ to itself.

\begin{theorem}\label{Thm3}
Assume that there is a map $\rho$ from $\mathcal{F}$ to $Map(A)$ defined by 
for any marked graph diagram $D$,
$\rho(D)=\rho_{D}$ such that 
$$\rho_{D}(P_{M}(D))=\rho_{D'}(P_{M}(D'))$$
where $D'$ is obtained from $D$ by applying once a positive or negative Yoshikawa move $\Gamma_{1}$ 
as depicted in Fig. \ref{RM1map} 
and $\rho$ is invariant under $\Gamma_{2}, \cdots, \Gamma_{8}$.
Then there exists a function $P_{M}^{\rho} : \mathcal{F} \rightarrow A$ 
defined by $P_{M}^{\rho}(D)=\rho_{D}(P_{M}(D))$ for each marked graph diagram $D$, 
which is invariant under all Yoshikawa moves.
\end{theorem}

\begin{proof}
Construct a function $P_{M}^{\rho} : {\mathcal F} \rightarrow A$ 
defined by $P_{M}^{\rho}(D)=\rho_{D}(P_{M}(D))$ for every marked graph diagram $D$,
where
$${\mathcal F} \rightarrow Map(A)\times A \rightarrow A$$ defined by
$$D \mapsto (\rho_{D}, P_{M}(D)) \mapsto \rho_{D}(P_{M}(D)).$$
By hypothesis, $P_{M}^{\rho}$ is invariant under $\Gamma_{1}$.
Since $\rho : {\mathcal F} \rightarrow Map(A)$ and $P_{M}$ are invariants under $\Gamma_{2}, \cdots, \Gamma_{8}$, 
$P_{M}^{\rho}$ is invariant for surface-links.
\end{proof}


\subsection{Examples of marked Kauffman bracket magmas and their invariants}

In this section, we will introduce some examples of marked Kauffman bracket magmas and their invariants.

For a given marked Kauffman bracket magma $(A; *, \bullet, \{a_{n}\})$, 
we can calculate the invariant values of any marked graph diagrams.
For example, 
let $0_{1}$ and $2_{1}^{-1}$ be marked graph diagrams of the standard sphere and the positive standard projective plane, respectively.
Then $P_{M}(0_{1})=1$ and 
$P_{M}(2_{1}^{-1})=(a_{2}*a_{1})\bullet (a_{1}*a_{2})$ as depicted in Fig. \ref{ExaCalculationP2}.

\begin{figure}[h!]
 \centering
 \includegraphics[width = 12cm]{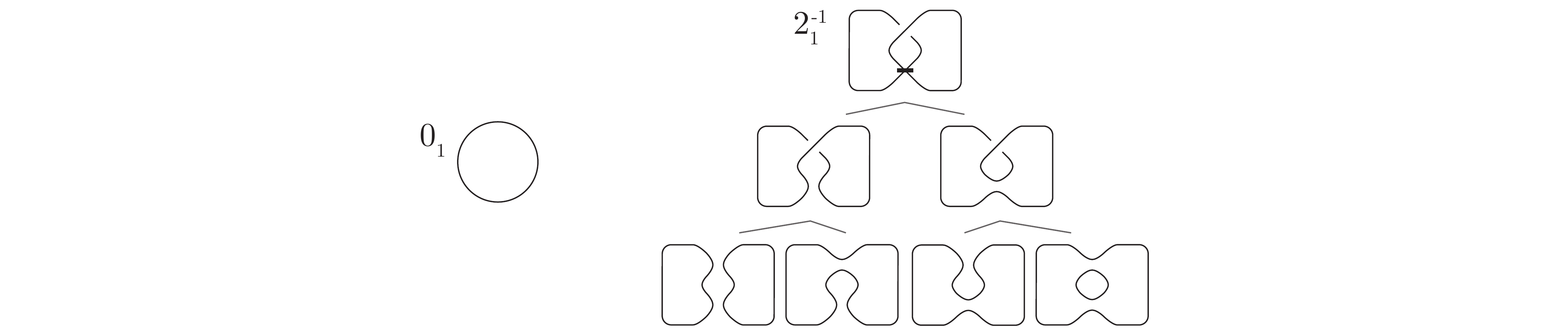}
 \caption{The calculation of $0_{1}$ and $2_{1}^{-1}$} \label{ExaCalculationP2}
\end{figure}

For the marked graph diagram $6_{1}^{0, 1}$ in Yoshikawa table, as shown in Fig. \ref{ExaCalculationMGD6},
$P_{M}(6_{1}^{0, 1})=(a_{2}\bullet P(D'))\bullet (a_{3}\bullet a_{2})$
where $D'$ is the classical diagram obtained from $6_{1}^{0,1}$ by splicing $x$ and $y$ at markers $m_{1}$ and $m_{2}$, respectively.

\begin{figure}[h!]
 \centering
 \includegraphics[width = 12cm]{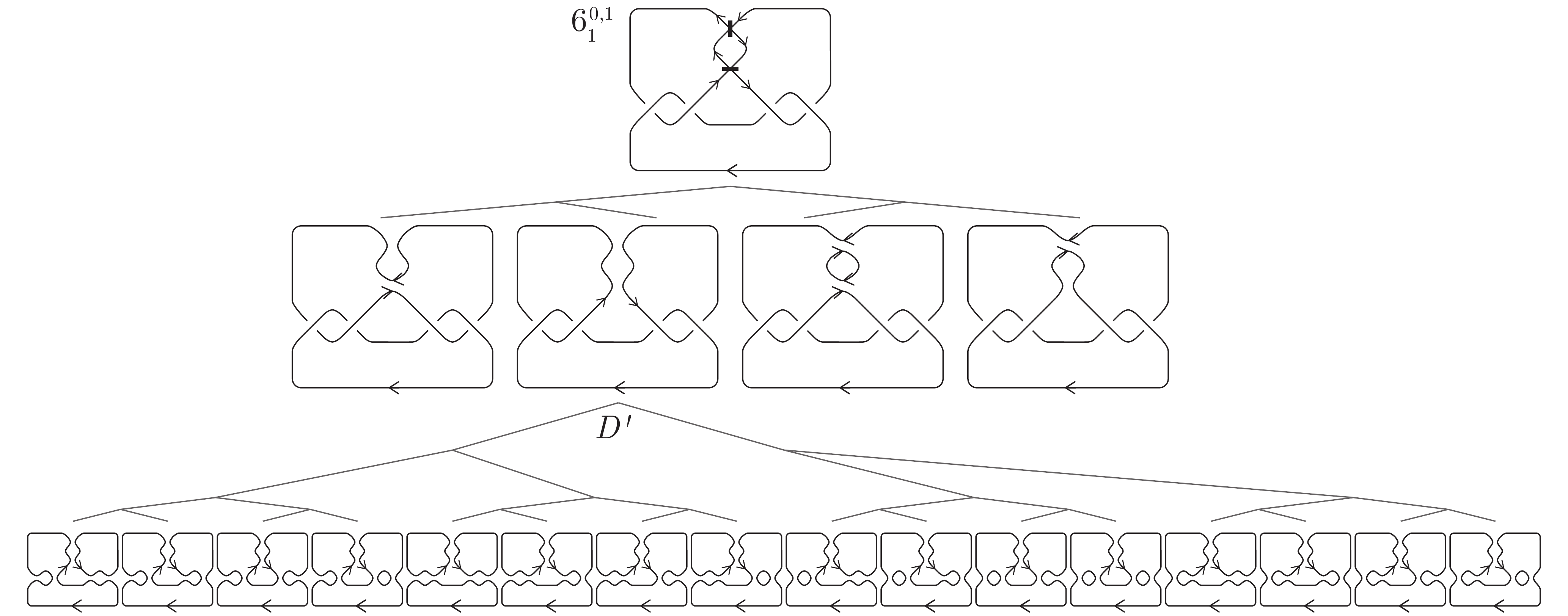}
 \caption{The calculation of $6_{1}^{0, 1}$} \label{ExaCalculationMGD6}
\end{figure}

\begin{example}\label{ExaMKB}
Let $\mathbb{Z}[p, r, x, y]$ be a polynomial ring equipped with two binary operations $*$ and $\bullet$ defined by 
$$a*b= pa+(-1-p)b+r,$$
$$a \bullet b = xa +yb.$$
For a fixed element $c \in \mathbb{Z}[p,r,x]$, take a sequence $a_{n} = c$ for each $n\in \mathbb{N}$.
Then the quadruple $(\mathbb{Z}[p,r,x]; *,\bullet, \{a_{n}\} )$ is a marked Kauffman bracket magma.
By Lemma \ref{LemGamma2345}, 
there exists the invariant $P_{M} : \mathcal{F} \rightarrow \mathbb{Z}[p, r, x, y]$ 
under $\Gamma_{2}$, $\Gamma_{3}$, $\Gamma_{4}$, $\Gamma_{4}'$ and $\Gamma_{5}$. 
For instance, 
$P_{M}(2_{1}^{-1})=$
$(x+y)(-c+r)$
and 
$P_{M}(6_{1}^{0, 1})=$
$c(x^{2}+2xy+y^{2})$ 
where $P(D')=c$.
\end{example}

\begin{example}\label{ExaMKB678}
In Example \ref{ExaMKB}, 
the marked Kauffman bracket magma $(\mathbb{Z}[p, r, x]; *,$ $\bullet,$ $\{a_{n}\} )$, obtained by substituting $y=1-x$, satisfies 
\begin{enumerate}
  \item $a_{n}= a_{n+2}$, for all $n\in \Neal$,
  \item $a_{n+1}\bullet a_{n}=a_{n}= a_{n}\bullet a_{n+1}$, for all $n\in \Neal$,
  \item $(a* b) \bullet (c * d) = (a \bullet c) * (b \bullet d)$ for all $a, b, c, d, \in \mathbb{Z}[p, r, x]$.
\end{enumerate}
By Theorem \ref{Thm2}, there exists the invariant $P_{M} : \mathcal{F} \rightarrow \mathbb{Z}[p, r, x]$ 
under Yoshikawa moves except $\Gamma_{1}$.

Define a map $\rho : \mathcal{F} \rightarrow \mathbb{Z}[p, r, x]$ by 
for any marked graph diagrams $D$, 
$\rho(D)=\rho_{D}$
where for any element $t \in \mathbb{Z}[p, r, x]$ 
\begin{equation*}
\rho_{D}(t) = \left\{
\begin{array}{cc} 
    t , & \text{if } c(D) \text{ is even}, \\
    -t+r ,& \text{if } c(D) \text{ is odd}, 
\end{array}\right.
\end{equation*}
and $c(D)$ is the number of classical crossings of $D$.
Since $\rho_{D}(P_{M}(D))=\rho_{D'}(P_{M}(D'))$
where $D'$ is obtained from $D$ by applying once a positive or negative Yoshikawa move $\Gamma_{1}$ 
as depicted in Fig. \ref{RM1map} 
and since $\rho$ is invariant under Yoshikawa moves except $\Gamma_{1}$,
the map $P^{\rho}_{M} : \mathcal{F} \rightarrow \mathbb{Z}[p, r, x]$ is invariant for surface-links by Theorem \ref{Thm3}.
Similar to Example \ref{ExaKB}, $P^{\rho}_{M}$ has only two values $1$ and $c$. 
\end{example}

The following example is a generalization of the Kauffman bracket magma in Example \ref{ExaKauffmanBracketInv}
and it corresponds to the Kauffman bracket polynomial. 

\begin{example}\label{ExaMarkedKauffmanBracketInv}
Let $\mathbb{Z}[A^{\pm 1}, x, y]$ be a polynomial ring equipped with two binary operations $*$ and $\bullet$ defined by 
$$a*b= Aa+A^{-1}b,$$
$$a \bullet b = xa +yb.$$
Consider the sequence $a_{n}=(-A^{2}-A^{-2})^{n-1}$ for each $n\in \mathbb{N}$.
Then the quadruple $(\mathbb{Z}[A^{\pm 1}, x, y] ; *, \bullet, \{a_{n}\})$ is a marked Kauffman bracket magma.
By Lemma \ref{LemGamma2345}, 
there exists the invariant $P_{M} : \mathcal{F} \rightarrow \mathbb{Z}[A^{\pm 1}, x, y]$ 
under $\Gamma_{2}$, $\Gamma_{3}$, $\Gamma_{4}$, $\Gamma_{4}'$ and $\Gamma_{5}$. 
For instance, $P_{M}(2_{1}^{-1})=$
$-A^{3}x-A^{-3}y$
and 
$P_{M}(6_{1}^{0, 1})=$
$(x^{2}+y^{2})(-A^{2}-A^{-2})+xy(A^{8}+A^{4}+4+A^{-4}+A^{-8})$ 
where $P(D')=A^{8}+2+A^{-8}$.
\end{example}

\begin{example}\label{ExaMKBInvAll}
Let $R$ be a quotient ring of $\mathbb{Z}[A^{\pm 1},x,y]$ by the ideal generated by
\begin{equation}\label{EqMKB-Kauffman}
 A^{4}+2+A^{-4} = x+ y(-A^{2}-A^{-2}) = y+ x(-A^{2}-A^{-2}) =1,
 \end{equation}
with the same binary operations and the sequence $a_{n}$ in Example \ref{ExaMarkedKauffmanBracketInv}.
Then $(R; *, \bullet, \{a_{n}\})$ is a marked Kauffman bracket magma.
From the equalities in (\ref{EqMKB-Kauffman}) with simple calculations, 
it follows that $(R; *,\bullet, \{a_{n}\})$ satisfies
\begin{enumerate}
  \item $a_{n}= a_{n+2}$, for all $n\in \Neal$,
  \item $a_{n+1}\bullet a_{n}=a_{n}= a_{n}\bullet a_{n+1}$, for all $n\in \Neal$,
  \item $(a* b) \bullet (c * d) = (a \bullet c) * (b \bullet d)$, for all $a,b,c,d \in R$.
\end{enumerate}
and by Theorem \ref{Thm2}, the map $P_{M} : \mathcal{F} \rightarrow A$ is an invariant under Yoshikawa moves except $\Gamma_{1}$. 

From now on, consider the set $\mathcal{F}^{ori}$ of oriented admissible marked graph diagrams 
and by abuse of notation, we denote $P_{M}|_{\mathcal{F}^{ori}}:\mathcal{F}^{ori} \rightarrow R$ by $P_{M}$.
Define a map $\rho : {\mathcal F}^{ori} \rightarrow Map(R)$ by for any diagram $D\in \mathcal{F}^{ori}$,
$\rho(D)=\rho_{D}$
where 
$$\rho_{D}(t)=(-A^{3})^{-p(D)+n(D)}t$$ 
for any element $t\in R$ 
and $p(D)$ (resp. $n(D)$) is the number of positive (resp. negative) classical crossings of $D$. 
It is easy to show that $\rho$ is invariant under Yoshikawa moves except $\Gamma_{1}$
and $\rho_{D}(P_{M}(D))=\rho_{D'}(P_{M}(D'))$
where $D'$ is obtained from $D$ by applying once a positive or negative Yoshikawa move $\Gamma_{1}$ 
as depicted in Fig. \ref{RM1map}.
Hence $P_{M}^{\rho} : \mathcal{F}^{ori} \rightarrow R$ is invariant under all Yoshikawa moves.
For the oriented marked graph diagram $6_{1}^{0, 1}$,  
$P_{M}(6_{1}^{0, 1})\equiv x+y$ 
and then 
$P_{M}^{\rho}(6_{1}^{0, 1})\equiv (-A^{3})^{0}(x+y)=x+y$.
\end{example}

\begin{example}\label{ExaMarkedSYLee}
Let $\mathbb{Z}[A^{\pm 1}, x, y, z, w]$ be a polynomial ring equipped with two binary operations $*$ and $\bullet$ defined by 
$$a*b= Aa+A^{-1}b,$$
$$a \bullet b = (x+Ay+A^{-1}z)a +(A^{-1}y+Az+w)b.$$
Consider the sequence $a_{n}=(-A^{2}-A^{-2})^{n-1}$ for each $n\in \mathbb{N}$.
Then the quadruple $(\mathbb{Z}[A^{\pm 1}, x, y, z, w] ; *, \bullet, \{a_{n}\})$ is a marked Kauffman bracket magma.
By Lemma \ref{LemGamma2345}, 
there exists the invariant $P_{M} : \mathcal{F} \rightarrow \mathbb{Z}[A^{\pm 1}, x, y, z, w]$ 
under $\Gamma_{2}$, $\Gamma_{3}$, $\Gamma_{4}$, $\Gamma_{4}'$ and $\Gamma_{5}$. 

For instance, $P_{M}(2_{1}^{-1})=-A^{3}x+(-A^{4}-A^{-4})y+(-A^{2}-A^{-2})z-A^{-3}w$
$=[[2_{1}^{-1}]]$
and 
$P_{M}(6_{1}^{0, 1})=$
$(-A^{2}-A^{-2})(x^{2}+w^{2})+(A^{8}+2+A^{-8})(y^{2}+z^{2})
+(A^{7}-A^{3}+2A^{-1}+A^{-5}+A^{-9})(xy+zw)
+(A^{9}+A^{5}+2A-A^{-3}+A^{-7})(xz+yw)
+(A^{2}+A^{-2})(A^{8}+A^{4}+A^{-4}+A^{-8})yz
+(A^{8}+A^{4}+4+A^{-4}+A^{-8})wx$.
\end{example}

The polynomial $[[ ~ ]]$ introduced by S. Y. Lee in Example \ref{ExaSYLee} is obtained from the invariant $P_{M}$ valued in the marked Kauffman magma in Example \ref{ExaMarkedSYLee}.


\subsection{Further discussion on Yoshikawa moves $\Gamma_{6}$ and $\Gamma_{6}'$}

The condition (2) of Theorem \ref{Thm2} is quite strong, 
because the condition require not only the commutativity of the binary operation $\bullet$, 
but also $a_{n} \bullet a_{n+1} =a_{n}$.
Therefore, instead of Theorem \ref{Thm2}, we can construct an invariant under Yoshikawa moves $\Gamma_{6}$ and $\Gamma_{6}'$ by using a specific map $\phi$.

\begin{theorem}\label{Thm4}
Let $P_{M} : \mathcal{F} \rightarrow A$ be the invariant 
valued in the marked Kauffman bracket magma $(A; *, \bullet, \{a_{n}\})$ in Theorem \ref{Thm1}.
Assume that there is an invariant $\phi : \mathcal{F} \rightarrow Map(A)$ 
under Yoshikawa moves except $\Gamma_{1}$, $\Gamma_{6}$ and $\Gamma_{6}'$
defined by for a diagram $D\in \mathcal{F}$,
$\phi(D)=\phi_{D}$ 
such that 
$$\phi_{D}(P_{M}(D))=\phi_{D'}(P_{M}(D'))$$
where $D'$ is obtained from $D$ by applying once one of Yoshikawa moves $\Gamma_{6}$ and $\Gamma_{6}'$.
Then there exists a function $P^{\phi}_{M} : \mathcal{F} \rightarrow A$ 
defined by $P^{\phi}_{M}(D)=\phi_{D}(P_{M}(D))$ for each marked graph diagram $D$, 
which is invariant under Yoshikawa moves except $\Gamma_{1}$.
\end{theorem}

\begin{remark}\label{RmkThm34}
Both the map $P_{M}^{\phi}$ of Theorem \ref{Thm4} and the map $P_{M}$ of Theorem \ref{Thm3} 
are invariant under Yoshikawa moves except $\Gamma_{1}$.
In Theorem \ref{Thm3}, we can use $P_{M}^{\phi}$ instead of $P_{M}$ 
and then there exists the invariant $(P_{M}^{\phi})^{\rho}$ under all Yoshikawa moves, 
denoted by $P_{M}^{\phi, \rho}$.
\end{remark}

\begin{example}\label{ExaMarkedKauffmanBracketInv-phi}
Let $I$ be the ideal of $\mathbb{Z}[A^{\pm 1},x,y]$ generated by
\begin{equation}\label{EqKauffman}
 A^{4}+2+A^{-4} =1,
\end{equation}
\begin{equation}\label{EqKauffman2}
 x+y(-A^{2}-A^{-2}) = y+x(-A^{2}-A^{-2}).
\end{equation}
Let $R$ be the localization of the quotient ring $\mathbb{Z}[A^{\pm 1},x,y]/I$ by the power of $x+y(-A^{2}-A^{-2})$
with the same binary operations and the sequence $a_{n}$ in Example \ref{ExaMarkedKauffmanBracketInv}. 
Then $(R; *, \bullet, \{a_{n}\})$ is a marked Kauffman bracket magma.
From Equation \ref{EqKauffman} with simple calculations, one can show that $(R; *, \bullet, \{a_{n}\})$ satisfies
\begin{enumerate}
  \item $a_{n}= a_{n+2}$, for all $n\in \Neal$,
  \item $(a* b) \bullet (c * d) = (a \bullet c) * (b \bullet d)$, for all $a,b,c,d \in R$,
\end{enumerate}
and by Theorem \ref{Thm2}, there exists an invariant $P_{M} : \mathcal{F} \rightarrow R$ under Yoshikawa moves except $\Gamma_{1}$, $\Gamma_{6}$ and $\Gamma_{6}'$.

Define a map $\phi : {\mathcal F} \rightarrow Map(R)$ by for any marked graph diagram $D$,
$\phi(D)=\phi_{D}$
where 
$$\phi_{D}(t)=(x+y(-A^{2}-A^{-2}))^{-m(D)}t =(y+x(-A^{2}-A^{-2}))^{-m(D)}t$$ 
for any element $t\in R$ and $m(D)$ is the number of marked vertices of $D$.
It is easy to see that $\phi$ is an invariant under Yoshikawa moves except $\Gamma_{1}$, $\Gamma_{6}$ and $\Gamma_{6}'$
and satisfies the condition of Theorem \ref{Thm4}.
By Theorem \ref{Thm4}, there exists the invariant $P_{M}^{\phi}$ under Yoshikawa moves except for $\Gamma_{1}$.

Consider the map $\rho : {\mathcal F}^{ori} \rightarrow Map(R)$ defined in Example \ref{ExaMKBInvAll} and
one can show that $P_{M}^{\phi, \rho} : \mathcal{F} \rightarrow  R$ defined by $P_{M}^{\phi, \rho}(D) = \rho_{D}(P_{M}^{\phi,}(D))$ is invariant under all Yoshikawa moves 
by Theorem \ref{Thm3} and Remark \ref{RmkThm34}.
Note that when we use the map $\phi$ corresponding to markers, we can relieve the condition $x+y(-A^{2}-A^{-2}) = y+x(-A^{2}-A^{-2})=1$ to $x+y(-A^{2}-A^{-2}) = y+x(-A^{2}-A^{-2})$.

\end{example}

\begin{remark}
In fact, if a marked Kauffman bracket magma satisfies the condition $a_{n} * a_{n+1} = a_{n+1} * a_{n} = a_{n}$, 
then $P_{M}$ becomes an invariant under $\Gamma_{1}$.
The condition $a_{n} * a_{n+1} = a_{n+1} * a_{n} = a_{n}$ is quite bad 
because it makes the invariant almost trivial. 
Note that since the relation $(a_{n} * a_{n+1})*(a_{n+1}* a_{n})=a_{n+1}$ holds, 
we obtain $a_{n} * a_{n} = a_{n+1}$ by using $a_{n} * a_{n+1} = a_{n+1} * a_{n} = a_{n}$. 
Let $c$ be a crossing of a marked graph diagram $D$. 
Suppose that we already spliced all crossings except for the crossing $c$. 
Then by splicing $c$, we obtain $n$ and $n+1$ trivial components or $n+1$ and $n$ trivial components by $A$ and $B$ splicings, respectively. 
It follows that we obtain $a_{n} *  a_{n+1}$ or $a_{n+1} *a_{n}$ in the last step for the calculation of $P_{M}(D)$. 
That is, if the condition $a_{n} * a_{n+1} = a_{n+1} * a_{n} = a_{n}$ is satisfied, 
then we can reduce the value $P_{M}(D)$ in $(A; *, \bullet, \{a_{n}\})$. 
Sequentially, by using $a_{n} * a_{n+1} = a_{n+1} * a_{n} = a_{n}$ and $a_{n} * a_{n} = a_{n+1}$, 
one can reduce the value $P_{M}(D)$ quite a lot. 
Then the condition $a_{n} * a_{n+1} = a_{n+1} * a_{n} = a_{n}$ is not preferred and it might be better to avoid it if it is possible.
\end{remark}

\section*{Acknowledgement}
The first author was supported by Basic Science Research Program through the National Research Foundation of Korea(NRF) funded by the Ministry of Education(2018R1D1A3B07044086 and 2021R1I1A1A01049100). 
The second author was funded by Russian Foundation for Basic Research (RFBR) grants 19-51-51004 and 20-51-53022.

\end{document}